\documentclass[12pt]{amsart}
\usepackage{amssymb}
\usepackage{amsfonts}
\usepackage{latexsym}
\usepackage{amscd}
\usepackage[mathscr]{euscript}
\usepackage{xy} \xyoption{all}
\numberwithin{equation}{section}

\newfont{\gd}{eufm10 scaled \magstep1}
\newfont{\gs}{eufm7 scaled \magstep1}
\newfont{\gss}{eufm5 scaled \magstep1}

\newcommand{\Z}{{\mathbb{Z}}}
\newcommand{\N}{{\mathbb{N}}}

\newcommand{\C}{{\mathbb{C}}}

\newcommand{\uloopr}[1]{\ar@'{@+{[0,0]+(-4,5)}@+{[0,0]+(0,10)}@+{[0,0] +(4,5)}}^{#1}}
\newcommand{\uloopd}[1]{\ar@'{@+{[0,0]+(5,4)}@+{[0,0]+(10,0)}@+{[0,0]+ (5,-4)}}^{#1}}
\newcommand{\dloopr}[1]{\ar@'{@+{[0,0]+(-4,-5)}@+{[0,0]+(0,-10)}@+{[0, 0]+(4,-5)}}_{#1}}
\newcommand{\dloopd}[1]{\ar@'{@+{[0,0]+(-5,4)}@+{[0,0]+(-10,0)}@+{[0,0 ]+(-5,-4)}}_{#1}}

\newcommand{\luloop}[1]{\ar@'{@+{[0,0]+(-8,2)}@+{[0,0]+(-10,10)}@+{[0, 0]+(2,2)}}^{#1}}

\newtheorem{lem}{Lemma}[section]

\newtheorem{theor}[lem]{Theorem}
\newtheorem{prop}[lem]{Proposition}
\newtheorem{rema}[lem]{Remark}
\newtheorem{defi}[lem]{Definition}

\newtheorem{exem}[lem]{Example}
\newtheorem{exems}[lem]{Examples}

\newtheorem{qtn}[lem]{Question}

\newtheorem{conjecture}[lem]{Conjecture}
\newtheorem*{nota}{Notation}

\newtheorem{noname}[lem]{}

\begin{document}
\title[K-theoretic characterization of graded isomorphisms]{Towards a K-theoretic characterization of graded isomorphisms between Leavitt path algebras}
\author{P. Ara}
\address{Departament de Matem\`atiques, Facultat de Ci\`encies, Edifici C, Universitat Aut\`onoma de Barcelona,
08193 Bellaterra (Barcelona), Spain.} \email{para@mat.uab.cat}
\author{E. Pardo}
\address{Departamento de Matem\'aticas, Facultad de Ciencias, Universidad de C\'adiz,
Campus de Puerto Real, 11510 Puerto Real (C\'adiz),
Spain.}\email{enrique.pardo@uca.es}
\urladdr{https://sites.google.com/a/gm.uca.es/enrique-pardo-s-home-page/}
\thanks{The first author was partially supported by the DGI
and European Regional Development Fund, jointly, through Project
MTM2011-28992-C02-01. The second author was partially supported by the DGI
and European Regional Development Fund, jointly, through Project
MTM2011-28992-C02-02, and by PAI III grants FQM-298 and P11-FQM-7156 of the
Junta de Andaluc\'{\i}a. Both authors are partially supported by 2009 SGR 1389 grant of the Comissionat per Universitats i Recerca de la Generalitat de Catalunya.} 
%\date{\today}
%
%\dedicatory{Version 2.3}
%\commby{}
%
\begin{abstract} In \cite{Hazrat},
 Hazrat gave a K-theoretic invariant for Leavitt
path algebras as graded algebras. Hazrat conjectured that this
invariant classifies Leavitt path algebras up to graded isomorphism,
and proved the conjecture in some cases. In this paper, we prove
that a weak version of the conjecture holds for all finite essential graphs.

\end{abstract}

\maketitle

\section{Introduction}

Leavitt path algebras of row-finite graphs have been introduced in
\cite{AA1} and \cite{AMFP}. They have become a subject of
significant interest, both for algebraists and for analysts working
in C*-algebras. The Cuntz-Krieger algebras $C^*(E)$ (the C*-algebra
counterpart of these Leavitt path algebras) are described in
\cite{Raeburn}. While sharing some striking similarities, the
algebraic and analytic theories present some remarkable differences,
as has been shown for instance in \cite{Tomforde}, \cite{AP},
\cite{AC} (see also \cite{Work}).

For a field $K$, the algebras $L_K(E)$ are natural generalizations of the algebras investigated by Leavitt in \cite{Le}, and are a specific type of
path $K$-algebras associated to a graph $E$ (modulo certain relations). The family of algebras which can be realized as the Leavitt path algebras of
a graph includes matrix rings ${\mathbb M}_n(K)$ for $n\in \mathbb{N}\cup \{\infty\}$ (where ${\mathbb M}_\infty(K)$ denotes matrices of countable
size with only a finite number of nonzero entries), the Toeplitz algebra, the Laurent polynomial ring $K[x,x^{-1}]$, and the classical Leavitt
algebras $L(1,n)$ for $n\ge 2$. Constructions such as direct sums, direct limits, and matrices over the previous examples can be also realized in
this setting.

A great deal of effort has been focused on trying to unveil the
algebraic structure of $L_K(E)$ via the graph nature of $E$.
Concretely, the literature on Leavitt path algebras includes
necessary and sufficient conditions on a graph $E$ so that the
corresponding Leavitt path algebra $L_K(E)$ is simple \cite{AA1},
purely infinite simple \cite{AA2} or exchange \cite{APS}. Another
remarkable approach has been the research of their monoids of
finitely generated projective modules $V(L_K(E))$ \cite{AMFP} and
the computation of its algebraic K-theory \cite{ABC}. The
availability of these data, and the tight connection between
properties of Leavitt path algebras and those of graph
$C^*$-algebras suggested the convenience of studying classification
results for (purely infinite simple) Leavitt path algebras using
K-theoretical invariants, in a similar way to that of \cite{Kirch,
Phil} for the case of Kirchberg algebras, and in particular for
purely infinite simple graph $C^*$-algebras. In this direction, important connections between
Leavitt path algebras, graph C*-algebras and symbolic dynamics have been recently
investigated by various authors, see e.g. \cite{AbAnhLouP2}, \cite{Hazrat2},
\cite{matsumoto} and \cite{matsumoto2}. In particular, substantial
advances on the above classification problem have been undertaken in \cite{AbAnhLouP2}, but at this
moment there is no complete classification result available  yet.

The graded structure of Leavitt path algebras has been an important tool
for studying the ideal structure of Leavitt path algebras, as well as
suitable versions of the Uniqueness Theorems for graph $C^*$-algebras in
this context \cite{Tomforde}. Also, it becomes relevant for the recent work of S.
Paul Smith on noncommutative geometry \cite{SPaulSmith}. Very recently,
Hazrat \cite{Hazrat, Hazrat2} proposed an approach to the
classification problem when the equivalence relation considered is
that of graded isomorphisms between algebras. After introducing a
suitable notion of graded $K_0$-group, he proved that for a strongly
$\Z$-graded algebra $A$ the groups $K_0^{gr}(A)$ and $K_0(A_0)$ have
a natural structure of $\Z[x,x^{-1}] $-module, and also that, thanks
to Dade's Theorem, we have $K_0^{gr}(A)\cong K_0(A_0)$ as ordered
$\Z[x,x^{-1}] $-modules. Hazrat conjectures in \cite[Conjecture
1]{Hazrat} that, given row-finite graphs $E$ and $F$,  we have that
 $L(E)\cong_{gr} L(F)$ if and only if there is an order-preserving
 $\Z[x,x^{-1}]$-module isomorphism $(K_0^{gr}(L(E)),
[1_{L(E)}])\cong (K_0^{gr}(L(F)), [1_{L(F)}]) $. Moreover, he proves
that the conjecture holds for finite graphs being acyclic, comet or,
more generally, polycephalic \cite[Theorem 9]{Hazrat}. \vspace{.2truecm}

In this paper
we prove that a weak version of Hazrat's conjecture holds for finite graphs with neither sources nor sinks.
Namely we prove that, under the above hypothesis on the graded $K_0$-groups, the algebra
$L(F)$ is graded-isomorphic to a certain deformation $L^g(E)$ of the algebra $L(E)$.
(See the beginning of Section \ref{MainSection} for the precise definition of
$L^g(E)$.)  We also obtain a description of the automorphisms of $L(E)$ which induce
the identity on $K_0^{{\rm gr}}$,  for the same class of graphs (Theorem \ref{theor:realizingiso}).
This is closely related to the uniqueness part of \cite[Conjecture 3]{Hazrat}.
\vspace{.2truecm}

The article is organized as
follows. Section \ref{Preliminares} includes the basic definitions
and examples that will be used throughout. In Section \ref{Hazrat
invariant} we will show that Hazrat's invariant is equivalent to
other K-theoretic invariants, and also to an invariant used for
classifying subshifts of finite type (see \cite{W}). This section also contains
lifting results of certain isomorphisms between the graded $K_0$-groups to
graded Morita equivalences and graded isomorphisms, respectively (see Theorem \ref{thm:gr-Morita} and Theorem \ref{thm:l=1}).
These results confirm the existence part of \cite[Conjecture 3]{Hazrat} in the case where
the isomorphism between the $K_0^{{\rm gr}}$-groups is induced by a strong shift equivalence (see also
\cite[Proposition 15(2)]{Hazrat2}).
In Section
\ref{MainSection} we state the main result of the paper and fix the
strategy to prove it. The hard part of the proof is contained in Section
\ref{ProofMainTheorem}. Section 6 contains a study of the uniqueness of liftings
of maps between the K-theoretic invariants.

\section{Preliminares}\label{Preliminares}

We briefly recall some graph-theoretic definitions and properties; more complete explanations and descriptions can be found in \cite{AA1}. A \emph{graph} $E=(E^0,E^1,r,s)$ consists of two sets $E^0,E^1$ and maps $r,s:E^1 \to E^0$.  (Some authors use the phrase `directed' graph for this structure.)  The elements of $E^0$ are called \emph{vertices} and the elements of $E^1$ \emph{edges}. And edge $e \in E^1$ is said to point from $s(e)$ to $r(e)$. If $s^{-1}(v)$ is a finite set for every $v\in E^0$, then the graph is called \emph{row-finite}.  A vertex $v$ for which $s^{-1}(v)$ is empty is called a \emph{sink};  a vertex $w$ for which $r^{-1}(w)$ is empty is called a \emph{source}.  If $F$ is a subgraph of $E$, then $F$ is called \emph{complete} in case $s^{-1}_F(v)  = s^{-1}_E(v)$  for every $v\in F^0$ having $s^{-1}_F(v) \neq \emptyset$.

A \emph{path} $\mu$ in a graph $E$ is a sequence of edges $\mu=e_1\dots e_n$ such that $r(e_i)=s(e_{i+1})$ for $i=1,\dots,n-1$.  In this case, $s(\mu):=s(e_1)$ is the \emph{source} of $\mu$, $r(\mu):=r(e_n)$ is the \emph{range} of $\mu$, and $n$ is the \emph{length} of $\mu$.  An edge $e$ is an {\it exit} for a path $\mu = e_1 \dots e_n$ if there exists $i$ such that $s(e)=s(e_i)$ and $e \neq e_i$. If $\mu$ is a path in $E$, and if $v=s(\mu)=r(\mu)$, then $\mu$ is called a \emph{closed path based at $v$}. If $\mu= e_1 \dots e_n$ is a closed path based at $v = s(\mu)$ and $s(e_i)\neq s(e_j)$ for every $i\neq j$, then $\mu$ is called a \emph{cycle}.

The following notation is standard.  Let $A$ be a $p\times p$ matrix
having non-negative integer entries (i.e., $A = (a_{ij})\in {\rm
M}_p(\Z^+)$).  The graph $E_A$ is defined by setting
$(E_A)^0=\{v_1,v_2,\ldots,v_p\}$, and defining $(E_A)^1$ by
inserting exactly $a_{ij}$ edges in $E_A$ having source vertex $v_i$
and range vertex $v_j$.  Conversely, if $E$ is a finite graph with
vertices $\{v_1,v_2,...,v_p\}$, then we define the {\it adjacency
matrix} $A_E$ {\it of} $E$ by setting $(A_E)_{ij}$ as the number of
edges in $E$ having source vertex $v_i$ and range vertex $v_j$.

\begin{defi}\label{definition}  {\rm Let $E$ be any row-finite graph, and $K$ any field. The {\em Leavitt path $K$-algebra} $L_K(E)$ {\em of $E$ with coefficients in $K$} is  the $K$-algebra generated by a set $\{v \mid v\in E^0\}$ of pairwise orthogonal idempotents, together with a set of variables $\{e,e^* \mid e \in E^1 \}$, which satisfy the following relations:

(1) $s(e)e=er(e)=e$ for all $e\in E^1$.

(2) $r(e)e^*=e^*s(e)=e^*$ for all $e\in E^1$.

(3) (The ``CK1 relations") \ $e^*e'=\delta _{e,e'}r(e)$ for all $e,e'\in E^1$.

(4) (The ``CK2 relations") \ $v=\sum _{\{ e\in E^1\mid s(e)=v \}}ee^*$ for every vertex $v\in E^0$ for which $s^{-1}(v)$ is
nonempty.
}
\end{defi}

When the role of the coefficient field $K$ is not central to the discussion, we will often denote $L_K(E)$ simply by $L(E)$.
The set $\{e^*\mid e\in E^1\}$ will be denoted by $(E^1)^*$. We let $r(e^*)$ denote $s(e)$, and we let $s(e^*)$ denote $r(e)$. If $\mu = e_1 \dots e_n$ is a path, then we denote by $\mu^*$ the
element $e_n^* \dots e_1^*$ of $L_K(E)$.

An alternate description of $L_K(E)$ is given in \cite{AA1}, where it is described in terms of a free
associative algebra modulo the appropriate relations indicated in Definition \ref{definition} above.
As a consequence, if $A$ is any $K$-algebra which contains a set of elements satisfying these same relations
(we call such a set an $E$-\emph{family}), then there is a (unique) $K$-algebra homomorphism from $L_K(E)$ to
$A$ mapping the generators of $L_K(E)$ to their appropriate counterparts in $A$.

If $E$ is a finite graph then $L_K(E)$ is unital, with $\sum _{v\in
E^0} v=1_{L_K(E)}$. Conversely, if $L_K(E)$ is unital, then $E^0$ is
finite. If $E^0$ is infinite then $L_K(E)$ is a ring with a set of
local units; one such set of local units consists of sums of
distinct elements of $E^0$.  There is a canonical $\Z$-grading on
$L_K(E)$, which is given by $L_K(E)=\bigoplus_{n\in \Z} L_K(E)_n$,
where, for each $n\in \mathbb{Z}$, the {\it degree} $n$ component
$L_K(E)_n$ is spanned by elements of the form $\{pq^* \mid {\rm
length}(p)-{\rm length}(q)=n\}$.  The set of \emph{homogeneous
elements} is $\bigcup_{n\in {\mathbb Z}} L_K(E)_n$, and an element
of $L_K(E)_n$ is said to be $n$-\emph{homogeneous} or
\emph{homogeneous of degree} $n$. The $K$-linear extension of the
assignment $pq^* \mapsto qp^*$ (for $p,q$ paths in $E$) yields an
involution on $L_K(E)$, which we denote simply as ${}^*$.
Information regarding the ``C$^*$-algebra of a graph", also known as
the ``Cuntz-Krieger graph C$^*$-algebra", may be found in
\cite{Raeburn}. In particular, the graph $C^*$-algebra $C^*(E)$ of a
graph $E$ is the completion of $L_{\C}(E)$ in a suitable norm
\cite{Raeburn}.\vspace{.2truecm}

Recall that for a ring $R$,  we denote by $K_0(R)$ the Grothendieck group of $R$.  This is the
group $F/S$, where $F$ is the free group generated by isomorphism classes of finitely generated
projective left $R$-modules, and $S$ is the subgroup of $F$ generated by symbols of the form
$[P\oplus Q]-[P]-[Q]$. As is standard, we denote the isomorphism class of $R$ in $K_0(R)$ by
$[1_R]$; we will call it the order-unit of the $K_0$-group.  The group $K_0(R)$ is the universal group of the monoid $V(R)$ of isomorphism classes of
finitely generated projective left $R$-modules (with binary operation in $V(R)$ given by
$[A]+[B]=[A\oplus B]$).

For a row-finite graph $E$, the {\it monoid of} $E$, denoted $M_E$, is the monoid generated by the
set $E^0$ of vertices of $E$ modulo appropriate relations, specifically,
$$M_E = \langle a_v, v\in E^0\mid a_v=\sum\limits
_{\{e\in s^{-1}(v)\}}a_{r(e)}, \, \text{ for }v\in E^0 \text{ which
is not a sink}  \rangle.$$ It is shown in \cite[Theorem 2.5]{AMFP}
that $V(L(E))\cong M_E$ for any row-finite graph $E$. This yields
$K_0(L(E))\cong \mbox{Grot}(M_E):=G$, where $\mbox{Grot}(M_E)$
denotes the universal group of the monoid $M_E$. If $E$ is finite,
then $M_E$ is finitely generated, hence so is its universal group
$G$. Thus $G$ admits a presentation $\pi: \mathbb{Z}^n\rightarrow G$
(an epimorphism). Here $\mbox{ker}(\pi)$ is the subgroup of
relations, which, in case $E$ does not have sinks, corresponds to
the image of the group homomorphism $I-A^t_E:
\mathbb{Z}^n\rightarrow\mathbb{Z}^n$, where $A^t_E$ is the transpose
of the incidence matrix $A_E$ of $E$. Hence we get
$$K_0(L(E))\cong G\cong \mathbb{Z}^n
/\mbox{ker}(\pi)=\mathbb{Z}^n/\mbox{im}(I-A^t_E)=\mbox{coker}(I-A^t_E).$$ Moreover, under this
isomorphism the element $[1_{L(E)}]$ is represented by $(1,1,...,1)^t + \mbox{im}(I-A^t_E)$ in
$\mbox{coker}(I-A^t_E)$.\vspace{.2truecm}

Now, we will recall the notion of fractional skew monoid ring
\cite{AGGP} in the particular case of the monoid being $\Z^+$. Let
$A$ be a unital ring, let $p^2=p\in A$ be an idempotent, and let
$\alpha: A\rightarrow pAp$ be an isomorphism. Then the fractional
skew monoid ring of $A$ over $\Z^+$ by $\alpha$ is the ring $R$
generated by an isomorphic copy of $A$ and two generators $t_+,
t_{-}$ satisfying the following relations:
\begin{enumerate}
\item $t_-t_+= 1$ and $t_+t_-= p$;
\item $t_+^na= \alpha^n(a)t_+^n$ for all $a\in A, n\in \Z^+$;
\item $at_-^n= t_-^n\alpha^n(a)$ for all $a\in A, n\in \Z^+$.
\end{enumerate}
We denote $R:=A[t_+,t_{-};\alpha]$. By \cite[2.2]{AGGP} the elements $r\in R$ can all
be written as `polynomials' of the form
$$r=a_{n}t_+^n+\ldots
+a_{1}t_++a_0+t_{-}a_{-1}+\ldots t_{-}^m a_{-m},$$
 with coefficients $a_i\in A$. By \cite[Proposition 1.6]{AGGP}, $R$
is a $\Z$-graded ring $R= \bigoplus_{i\in\Z} R_i$, with $R_i=
At^i_+$ for $i>0$ and $R_i= t_-^{-i}A$ for $i<0$, while $A_0=A$.
This construction is an exact algebraic analog of the construction
of the crossed product of a C*-algebra by an endomorphism introduced
by Paschke \cite{Paschke}. In fact, if $A$ is a C*-algebra and the
corner isomorphism $\alpha$ is a *-homomorphism, then Paschke's
C*-crossed product, which he denotes by $A\rtimes _{\alpha}\N$, is
just the completion of $A[t_+,t_{-};\alpha]$ in a suitable norm. If
$E$ is a finite graph with no sources, then it is known that
$L(E)=L(E)_0[t_+, t_{-};\alpha]$ for suitable elements $t_+,t_{-}\in
L(E)$, being $\alpha$ the corner isomorphism of $L(E)_0$ defined by
the rule $\alpha (a)=t_+at_{-}$ for all $a\in L(E)_0$
\cite{AGGP}.

\section{Equivalent forms of Hazrat's invariant}\label{Hazrat invariant}
\label{sect:Equival}

In this section, we will relate the invariant proposed by Hazrat for
classifying Leavitt path algebras up to graded isomorphism to a
couple of K-theoretic invariants, and to an invariant (the so-called
dimension triple) which appears in the classification of subshifts
of finite type.\vspace{.2truecm}

Let $\Gamma $ be a group and let $A$ be a $\Gamma $-graded ring. We will denote by $\mbox{Gr}-A$ the category of
graded $A$-modules and graded homomorphisms, see e.g. \cite{HazMon}.
A graded right
$A$-module is projective in the category $\mbox{Mod}-A$ of right $A$-modules
if and only if it is projective in the category $\mbox{Gr}-A$ of graded right $A$-modules
(see e.g.  \cite[Proposition 1.2.13]{HazMon}).
Recall from \cite[2.5]{Hazrat} that  if $\Gamma$ is a group and $A$ is a
$\Gamma$-graded ring, then the graded Grothendieck group $K_0^{gr}(A)$ is
defined as the universal enveloping group of the abelian monoid of
isomorphism classes of graded finitely generated projective modules
over $A$.
Also, for $g \in \Gamma$, the $g$-suspension functor
$\mathcal{T}_{g}: \mbox{Gr}-A\rightarrow \mbox{Gr}-A$ (defined by
the rule $(\mathcal{T}_{g}(M))_h=M_{hg}$) is an isomorphism with the
property that $\mathcal{T}_{g}\mathcal{T}_{h}=\mathcal{T}_{gh}$ for
$g,h \in \Gamma$. Moreover,  $\mathcal{T}_{g}$ restricts to the category
of graded finitely generated projective $A$-modules, and hence it
induces a structure of $\Z[\Gamma]$-module on $K_0^{gr}(A)$ by the rule
$g[P]:=[\mathcal{T}_{g}(P)]$.

By Dade's Theorem \cite[Theorem 3.1.1]{Nasta_VanOs}, if $\Gamma$ is a
group and $A$ is a strongly $\Gamma$-graded ring, then the functor
$(-)_0: \mbox{Gr}-A\rightarrow \mbox{Mod}-{A_0}$ (defined by the
rule $M\mapsto M_0$) is an additive functor with inverse $-
\otimes_{A_0} A : \mbox{Mod}-A_0 \rightarrow \mbox{Gr}-A$, so that
it induces an equivalence of categories. Furthermore, since
$A_{g}\otimes_{A_0} A_{h}\cong A_{gh}$ for $g,h \in \Gamma$ as
$A_0$-bimodules, the suspension functor $\mathcal{T}_{g}:
\mbox{Gr}-A\rightarrow \mbox{Gr}-A$ induces a functor
$\mathcal{T}_{g}:  \mbox{Mod}-A_0\rightarrow \mbox{Mod}-A_0$, given
by the rule $M\mapsto M\otimes_{A_0}A_{g}$, such that
$\mathcal{T}_{g}\mathcal{T}_{h}=\mathcal{T}_{gh}$ for $g,h \in \Gamma$
making the following diagram commutative:
$$\xymatrix{\mbox{Gr}-A\ar[rr]^{\mathcal{T}_{g}} \ar[dd]_{(-)_0} & & \mbox{Gr}-A \ar[dd]^{(-)_0}\\
 & & \\
\mbox{Mod}-A_0\ar[rr]_{\mathcal{T}_{g}} & & \mbox{Mod}-A_0}
$$
Therefore $K_i(A_0)$ is also a $\Z[\Gamma]$-module, and $K_i^{gr}(A)\cong
K_i(A_0)$ as $\Z[\Gamma]$-modules for $i\geq 0$.

In the case of Leavitt path algebras, we will specialize the above
situation to the natural $\Z$-grading, so that $\Z[\Gamma]$ will be
$\Z[x,x^{-1}]$, and in particular we only need to take care of
$\mathcal{T}_1$ (which represents $x\cdot$) to have a complete
picture of the action.

Recall that two $\Gamma$-graded rings $A$ and $B$ are said to be {\it graded Morita equivalent}
in case there exists an equivalence of categories $F\colon  \mbox{Gr}-A\to\mbox{Gr}-B$ such that
$F\circ\mathcal T _g= \mathcal T _g \circ F$ for all $g\in \Gamma $. By \cite[Theorem 5.4]{GreenGordon},
$A$ and $B$ are graded Morita equivalent if and only if there is a graded right $A$-module $P$, which is a finitely generated
projective generator in $\mbox{Mod}-A$, such that $\mbox{End}_A(P)$ and $B$ are isomorphic
as graded rings.

Given a graded equivalence $F\colon  \mbox{Gr}-A\to\mbox{Gr}-B$, there is an induced ordered $\Z[\Gamma]$-module homomorphism
$K_0^{gr}(F)\colon K_0^{gr}(A)\to K_0^{gr}(B)$ defined by $K_0^{gr}(F)([X])= [F(X)]$ for a finitely generated graded projective module $X$.
Note that $K_0^{gr}(F)=K_0^{gr}(G)$ if $F\cong G$. If $\phi \colon A\to B$ is a graded isomorphism then $\phi$ induces an obvious graded Morita equivalence
$\Phi$ such that $K_0^{gr}(\Phi) =K_0^{gr}(\phi)$.

\begin{exems}
 \label{exam:grMoritaequi}{\rm
 (1) Let $e$ be a homogeneous full idempotent in the $\Gamma $-graded ring $A$, then $A$ is graded Morita equivalent to $eAe$, cf. \cite[Example 2.3.1]{HazMon}.

  (2) Assume that $A, B$ are $\Gamma $-graded rings and $F\colon  \mbox{Gr}-A\to\mbox{Gr}-B$ is a graded equivalence. Set
  $P=F(A_A)$. Then $P$ has a natural structure of graded $A-B$-bimodule, and $F\cong -\otimes _{A} P$, see e.g. \cite[Theorem 2.3.5]{HazMon}.
  Now assume that $P_B\cong B_B$ as graded modules. We can give to $B$ a structure of $A-B$-bimodule such that the above is a graded
  isomorphism of bimodules. Therefore $F\cong -\otimes _A B$, and in particular we obtain a graded ring isomorphism $\phi $ as the composition
  of the graded isomorphisms
  $$ A\to \text{End}_A(A) \to \text{End}_B (A\otimes _A B)\to \text{End}_B(B) \to B .$$
  It is easily checked that the Morita equivalence induced by $\phi$ is exactly $-\otimes _A B$. Therefore
  $$K_0^{gr}(F) = K_0^{gr} (-\otimes _{A} P)= K_0^{gr}(-\otimes _A B)= K_0^{gr} (\phi).$$}
  \end{exems}

We now state Hazrat's main conjecture.

\begin{conjecture}
\label{conj:Hazrat}{\rm \cite[Conjecture 1]{Hazrat} Given finite
graphs $E$ and $F$, we have that $L(E)\cong_{gr} L(F)$ if and only
if $$(K_0^{gr}(L(E)), [1_{L(E)}])\cong_{\Z[x,x^{-1}]}
(K_0^{gr}(L(F)), [1_{L(F)}]) .$$}
\end{conjecture}

\medskip

Hazrat proves that the conjecture holds for acyclic, comet and
polycephalic finite graphs \cite[Theorem 9]{Hazrat}. The isomorphism appearing in the Conjecture
is assumed to preserve the order in both directions, to commute with the action of $\mathcal T _*$, and to take the canonical order-unit
$[1_{L(E)}]$ of $K_0^{gr}(L(E))$ to the canonical order-unit $[1_{L(F)}]$ of $K_0^{gr}(L(F))$. These are the isomorphisms in the
category of triples $(G,G^+,u)$, where $G$ is a $\Z[x,x^{-1}]$-module, $G^+$ is a positive cone in $G$, and $u$
is a distinguished order-unit in $G$ (see \cite[3.5]{HazMon}). This will be referred to later as the order-unit preserving category.

\vspace{.2truecm}

Because our tools require the use of the structure of strongly $\Z$-graded rings, we will
need to ask our graphs to contain no \texttt{sink}s \cite[Theorem
4]{Hazrat}. Also, we will consider a different picture of the
$\Z$-graded structure of a Leavitt path algebra $L(E)$ over a finite
graph $E$, which links grading with the existence of special non
necessarily unital endomorphisms of $L(E)_0$. To be concrete, assume
that $E$ is a finite graph that contains no sources. If $E^0=\{
v_1,\dots ,v_n\}$, then for each $1\leq i\leq n$ there exists at
least one $e_i\in E^1$ such that $r(e_i)=v_i$. If we define
$t_+=\sum\limits_{i=1}^{n}e_i\in L(E)_1$ and
$t_{-}=\sum\limits_{i=1}^{n}e_i^*\in L(E)_{-1}$, then it is easy to
see that $t_{-}t_+=1$. Hence, by \cite[Lemma 2.4]{AGGP},
$$L(E)=L(E)_0[t_+, t_{-}; \alpha].$$
Since this picture of $L(E)$ plays a central role in our arguments, from now on we will ask our
graphs to have no \texttt{source}s
Moreover, if  $p:=t_+t_{-}$, then $\alpha: L(E)_0\rightarrow
pL(E)_0p$ is a corner isomorphism defined by the rule $\alpha
(a)=t_+at_{-}$. Under this picture,
$L(E)_n=L(E)_0t_+^n=L(E)_0p_nt_+^n$ and
$L(E)_{-n}=t_{-}^nL(E)_0=t_{-}^np_nL(E)_0$ for $n>0$, where
$p_n=\alpha^n(1)$. Moreover, since $L(E)_0$ is a unital
ultramatricial algebra (see e.g. \cite[Proof of Theorem 5.3]{AMFP}),
we can classify it via its Bratteli diagram $X_A$. This Bratteli
diagram can be described fixing at each level of the diagram the
same number of vertices, namely $\vert E^0\vert$ vertices at each
level, while the edges between two different levels are codified by
the entries of the adjacency matrix $A:=A_E$ of the graph $E$.

\medskip

The following standard definition will be useful in order to shorten our statements.

\begin{defi}
 \label{def:essential-graphs} {\rm An {\it essential graph} is a graph with neither sources nor sinks.}
\end{defi}

We will show that the K-theoretic part of \cite[Conjecture
1]{Hazrat} can be rephrased in equivalent forms which are more
useful to relate the K-theoretic isomorphisms with the behavior of
the graded structure of the Leavitt path algebras involved.

Given any $M\in \mbox{Mod}-L(E)_0$, we will define a right $L(E)_0$-module $M^{\alpha}$ as follows:
\begin{enumerate}
\item As abelian group, $M^{\alpha}:=Mp$;
\item The action is defined by the rule
$$\xymatrix @R=.5pc @C=.5pc { M^{\alpha} \times L(E)_0 \ar[r]  & M^{\alpha}\\
(m,a) \ar @{|->}[r]  & m\ast a:= m\alpha (a)
}
$$
\end{enumerate}
Notice that, whenever $m\in M^{\alpha}$, $m\ast 1=m\alpha(1)=m\cdot
p=m$, so that $M^{\alpha}$ is a unital module.

\begin{lem}\label{Lemma: FirstEquiv}
The right $L(E)_0$-modules $M^{\alpha}$ and $M\otimes_{L(E)_0} L(E)_0t_+$ are isomorphic.
\end{lem}
\begin{proof}
Since $pt_+=t_+$, the module $M\otimes_{L(E)_0} L(E)_0t_+$ is
unital. Moreover,
$$M\otimes_{L(E)_0} L(E)_0t_+=M\otimes_{L(E)_0} L(E)_0pt_+.$$
Hence, the map
$$\xymatrix @R=.5pc @C=.5pc { **[l]M\otimes_{L(E)_0} L(E)_0pt_+ \ar[rrr]^(.5)f & & & M^{\alpha}\\
m\otimes (xp)t_+ \ar @{|->}[rrr] & & & mxp }
$$
is a group isomorphism. Also, for any $r\in L(E)_0$,
$$
f((m\otimes (xp)t_+)r)=f(m\otimes (xp)\alpha(r)t_+)=(mxp)\alpha(r)=(mxp)\ast
r=f(m\otimes (xp)t_+)\ast r
$$
which ends the proof.
\end{proof}

Define $\mathcal{T}:=\mathcal{T}_1: \mbox{Mod}-L(E)_0\rightarrow
\mbox{Mod}-L(E)_0$ by
$$\mathcal T (M)= M\otimes_{L(E)_0} L(E)_1 = M\otimes _{L(E)_0} L(E)_0t_+.$$
This corresponds to the translation functor when $L(E)$ is strongly graded, i.e., when $E$ has no sinks
(\cite[Theorem 4]{Hazrat}).
By Lemma \ref{Lemma: FirstEquiv}, the induced
isomorphism
$$\mathcal{T}_{\ast}: K_0(L(E)_0)\rightarrow K_0(L(E)_0)$$
is given by the rule $\mathcal{T}_{\ast}([P])=[P^{\alpha}]$.\vspace{.2truecm}

We will make use of the following well-known (and easy to prove)
lemma:

\begin{lem}
\label{lem:full=no-sinks} Let $F$ be a finite graph. Then $p=t_+t_-$
is a full idempotent in $L(F)_0$ if and only if $F$ does not have
sinks.
\end{lem}

Assume that $E$ is an essential graph. It follows
from Lemma \ref{lem:full=no-sinks} and the fact that $\alpha:
L(E)_0\rightarrow pL(E)_0p$ is an isomorphism that the induced map
$$\alpha_{\ast}: K_0(L(E)_0)\rightarrow K_0(L(E)_0)$$
is a group isomorphism.

\begin{lem}\label{Lemma: SecondEquiv} If $E$ is a finite essential graph, then the maps $\alpha_{\ast}$ and
$\mathcal{T}_{\ast}$ are mutually inverse isomorphisms.
\end{lem}
\begin{proof}
Since both maps are group isomorphisms, it suffices to show that
$$\mbox{Id}_{K_0(L(E)_0)}=\mathcal{T}_{\ast}\circ \alpha_{\ast}.$$
Let $e^2=e\in L(E)_0$. Then
$$(\mathcal{T}_{\ast}\circ \alpha_{\ast})([eL(E)_0])=\mathcal{T}_{\ast}([\alpha(e)L(E)_0])=[(\alpha(e)L(E)_0)^{\alpha}].$$
The map
$$ \xymatrix @R=.5pc @C=.5pc { eL(E)_0 \ar[rrr]^(.4)g & & & (\alpha(e)L(E)_0)^{\alpha}\\
x \ar @{|->}[rrr] & & & \alpha(x)
}
$$
is well-defined and additive. As $\alpha :L(E)_0\rightarrow
pL(E)_0p$ is a corner isomorphism, $g$ is injective. Now, let $y\in
(\alpha(e)L(E)_0)^{\alpha}$. Thus, $y=\alpha(e)rp$ for some $r\in
L(E)_0$. Then,
$$y=\alpha(e)rp=\alpha(e\cdot 1)rp=\alpha(e)prp $$
and since $prp\in pL(E)_0p=\alpha(L(E)_0)$, there exists a unique $s\in L(E)_0$ such that $\alpha(s)=prp$, whence
$$\alpha(e)prp=\alpha(e)\alpha(s)=\alpha(es) $$
and hence $y=g(es)$, so that $g$ is onto.

Finally, for any $x\in eL(E)_0$ and any $r\in L(E)_0$,
$$g(xr)=\alpha(xr)=\alpha(x)\alpha(r)=\alpha(x)\ast r=g(x)\ast r,$$

Thus $[eL(E)_0]=[(\alpha(e)L(E)_0)^{\alpha}]$ in $K_0(L(E)_0)$, so
that $(\mathcal{T}_{\ast}\circ \alpha_{\ast})([eL(E)_0])=[e(L(E)_0]$
for all idempotents $e$ in $L(E)_0$. Since $K_0(L(E)_0)$ is
generated by the elements $[eL(E)_0]$, for $e=e^2\in L(E)_0$, we
conclude that $\mathcal T_{\ast}\circ \alpha
_{\ast}=\mbox{Id}_{K_0(L(E)_0)}$.
\end{proof}

As a consequence we have the following result.

\begin{prop}\label{Corollary: TheGoodOne}
Let $E,F$ be finite essential graphs, and
$\alpha, \beta$ the respective corner isomorphisms. Then, the
following are equivalent:
\begin{enumerate}
\item There exists an isomorphism $\Psi: (K_0(L(E)_0), K_0^+(L(E)_0))\to (K_0(L(F)_0), K_0^+(L(F)_0))$
such that $\Psi\circ
(\mathcal{T}_E)_{\ast}=(\mathcal{T}_F)_{\ast}\circ\Psi$
\item There exists an isomorphism $\Phi: (K_0(L(E)_0), K_0^+(L(E)_0))\to (K_0(L(F)_0), K_0^+(L(F)_0))$ such that $\Phi\circ \alpha_{\ast}=\beta_{\ast}\circ \Phi$
\item There exists an isomorphism of ordered $\Z[x,x^{-1}]$-modules $K_0^{gr}(L(E))\cong K_0^{gr}(L(F))$.
\end{enumerate}
Moreover, the result holds also in the order-unit
preserving category. \hspace{\fill}$\Box$
\end{prop}

\begin{nota}
{\rm We will denote the condition in Proposition \ref{Corollary: TheGoodOne}(1) as
$$(K_0(L(E)_0), K_0^+(L(E)_0))\cong_{\Z[x,x^{-1}]} (K_0(L(F)_0), K_0^+(L(F)_0)).$$
and that in Proposition \ref{Corollary: TheGoodOne}(2) as
$$(K_0(L(E)_0), K_0^+(L(E)_0), \alpha_{\ast})\cong (K_0(L(F)_0), K_0^+(L(F)_0), \beta_{\ast}).$$}
\end{nota}

In order to give the next equivalent condition, we need to recall some facts about conjugacy of subshifts of finite type \cite{W}.
Recall \cite[Definition 7.3.1]{L-M} that given $A,B$ nonnegative integral square matrices, and $\ell \geq 1$, a shift equivalence of
lag $\ell$ from $A$ to $B$ is a pair of rectangular nonnegative integral matrices $(S,R)$ satisfying the four shift equivalence equations:
\begin{enumerate}
\item (i) $AR=RB$; (ii) $SA=BS$.
\item (i) $A^{\ell}=RS$; (ii) $B^{\ell}=SR$.
\end{enumerate}
We will denote this equivalence by $A\sim_{SE} B$.

In order to relate this equivalence with our problem, recall that if
$E$ is a finite graph with no sinks, and $A:=A_E^t$ denotes the
transpose of the adjacency matrix, then the ultramatricial algebra
$L(E)_0$ is a direct limit of matricial algebras $L(E)_{0,n}$
($n\geq 0$) whose connecting maps are induced by the (CK2) relation
(see e.g. \cite[Proof of Theorem 5.3]{AMFP}). Moreover, by the
picture of the Bratteli diagram associated to this ultramatricial
algebra, $K_0(L(E)_0)\cong \varinjlim (K_0(L(E)_{0,n}), A)$, where
$K_0(L(E)_{0,n})\cong \Z^N$ for all $n\geq 0$, and $N=\vert
E^0\vert$. We will denote $G_m:=K_0(L(E)_{0,m})$ ($m\geq 0$) and
$G:=K_0(L(E)_{0})$. Then, we have a commutative diagram
$$
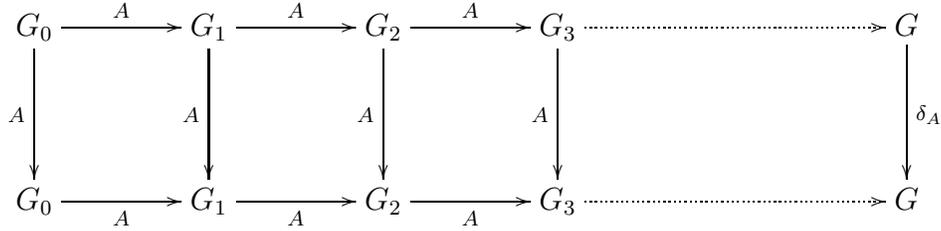
\xymatrix@!=3.5pc{G_0 \ar[r]^A \ar[d]_A & G_1 \ar[r]^A \ar[d]_A & G_2 \ar[r]^A \ar[d]_A &G_3  \ar[d]_A  \ar @{.>}[rr] & &  G\ar[d]^{\delta_A}\\
G_0 \ar[r]_A  & G_1 \ar[r]_A  & G_2 \ar[r]_A  &G_3  \ar @{.>}[rr] & &  G   }
$$
that induces a group isomorphism $\delta_A:G\rightarrow G$ by multiplication by $A$ on the direct limit.

With that in mind, we can show the following result

\begin{lem}\label{Lemma: ThirdEquiv}
Let $E$ be a finite essential graph, let $A$ be
the transpose of the incidence matrix of $E$, and let $\delta_A$ and
$\alpha_{\ast}$ be the automorphisms of $K_0(L(E)_0)$ defined above.
Then, $\alpha_{\ast}$ is the inverse of $\delta_A$.
\end{lem}
\begin{proof}
Let $G=\varinjlim (G_n, A)$, where $G_n=\Z^{\vert E^0\vert}$ for all $n\geq 0$.
The map corresponding to $\alpha: L(E)_0\rightarrow L(E)_0$ at each step of the direct limit is
the identity map $\mbox{id}: G_k\rightarrow G_{k+1}$. Hence, we have a commutative diagram
$$
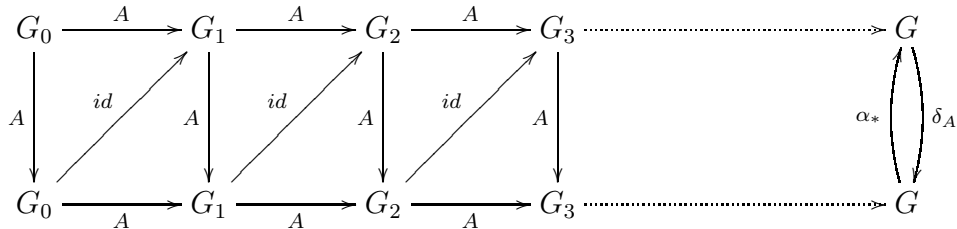
\xymatrix@!=3.5pc{G_0 \ar[r]^A \ar[d]_A & G_1 \ar[r]^A \ar[d]_A & G_2 \ar[r]^A \ar[d]_A &G_3  \ar[d]_A  \ar @{.>}[rr] & &  G\ar @/^/[d]^{\delta_A}\\
G_0 \ar[r]_A \ar[ru]^{id} & G_1 \ar[r]_A \ar[ru]^{id} & G_2 \ar[r]_A \ar[ru]^{id} &G_3  \ar @{.>}[rr] & &  G \ar @/^/[u]^{\alpha_*}  }
$$
Since both maps are automorphisms of $K_0(L(E)_0)$, it suffices to
show that $\delta_A \circ \alpha_{\ast}=\mbox{Id}_G$. Following the
diagram,
$$(\delta_A \circ \alpha_{\ast})(\iota_{n, \infty}(g))=
\delta_A (\alpha((\iota_{n, \infty}(g))))=\delta_A(\iota_{n+1, \infty}(g))=\iota_{n+1, \infty}(Ag)=\iota_{n+1,\infty}\iota_{n,n+1}(g)=\iota_{n, \infty}(g),$$
where $\iota_{n, \infty}\colon G_n \to G$ is the canonical map.
This concludes the proof.
\end{proof}

\begin{rema}
\label{rem:T=deltaA} {\rm It follows from Lemmas \ref{Lemma:
SecondEquiv} and \ref{Lemma: ThirdEquiv} that $\mathcal
T_*=\delta_A$ when $E$ is a finite essential graph. In fact, it is easy to prove directly that this equality holds
for any finite graph without sinks: Setting $A_i=L(E)_i$, we have, for $v\in E^0$,
$$\mathcal T _*([vA_0])= [vA_0\otimes _{A_0} A_1]=[vA_1]=\sum _{e\in s^{-1}(v)} [eA_0]=\sum _{e\in s^{-1}(v)} [r(e)A_0].
$$
Similarly , if $| \gamma |\ge 1$, we have
$$
\mathcal T _*([\gamma \gamma^*A_0])  =[\gamma\gamma^*A_1]= \sum _{e\in s^{-1}(r(\gamma))} [\gamma ee^*\gamma ^*A_1]
 = \sum _{e\in s^{-1}(r(\gamma))} [\gamma ' ee^*(\gamma') ^*A_0] \, , $$
where $\gamma = f\gamma '$ for $f\in E^1$ and $| \gamma ' | =| \gamma | -1$. }
\qed\end{rema}

\medskip

It is important to notice that
$$(K_0(L(E)_0), K_0^+(L(E)_0),  \delta_{A_E^t})\cong (K_0(L(F)_0), K_0^+(L(F)_0),  \delta_{A_F^t})$$
is equivalent to $A_E^t\sim_{SE} A_F^t$ with lag $l$ for some $l \geq 1$ \cite[Theorem 7.5.8]{L-M}. Now, we can state the main result of this section

\begin{theor}\label{Theorem: MainFirstSection}
Let $E,F$ be finite essential graphs, let
$\alpha, \beta$ be the respective corner isomorphisms. Set $A= A_E^t$ and $B= A_F^t$, and let $\delta_A$,
$\delta_B$ be the automorphisms of $K_0(L(E)_0)$ and $K_0 (L(F)_0)$ defined above. Then,
the following are equivalent:
\begin{enumerate}
\item $A\sim_{SE}B$.
\item $ (K_0(L(E)_0), K_0^+(L(E)_0))\cong_{\Z[x,x^{-1}]} (K_0(L(F)_0), K_0^+(L(F)_0))$.
\item $(K_0(L(E)_0), K_0^+(L(E)_0), \alpha_{\ast})\cong (K_0(L(F)_0), K_0^+(L(F)_0), \beta_{\ast})$.
\item $(K_0(L(E)_0), K_0^+(L(E)_0), \delta_A)\cong (K_0(L(F)_0), K_0^+(L(F)_0), \delta_B)$.
\item There is an ordered  $\Z[x,x^{-1}]$-module isomorphism $K_0^{gr}(L(E))\cong
K_0^{gr}(L(F))$.
\end{enumerate}
Moreover, the equivalences also hold in the order-unit
preserving category. \hspace{\fill}$\Box$
\end{theor}

For a non-negative matrix $A\in M_N(\Z^+)$, we denote by $(\Delta_A, \Delta_A^+, \delta _A)$ the dimension triple of $A$.
Here $\Delta _A= \varinjlim (G_n, A)$, with $G_n=\Z ^N$ for all $n$, and $\delta _A$ is induced by multiplication by $A$ at each level, so that
$K_0(L(E)_0)\cong \Delta_A$ and the dimension triple is isomorphic to $(K_0(L(E)_0), K_0^+(L(E)_0), \delta_A)$, where $E$ is the graph with incidence matrix $A^t$.

Note that every shift equivalence $(S,R)\colon A\sim_{SE} B$ of lag $l$, where $A\in M_N(\Z^+)$ and $B\in M_M(\Z^+)$,
induces an isomorphism $\Phi$ of dimension triples such that $ \Phi (\iota_{n,\infty}(x))= \iota_{m+n,\infty} (Sx)$ for $x\in G_n=\Z^N$, where $m$ is a fixed non-negative integer.
The inverse map $\Psi \colon \Delta _B\to \Delta _A$ is defined by $\Psi (\iota _{m+n,\infty}(x)) = \iota_{n+l,\infty} (Rx)$ for $x\in \Z^M$ and $n\ge 0$.
(See the diagram on \cite[page 234]{L-M} for an illustrative picture in the case $l=2$.)
Conversely, given an isomorphism of dimension triples $ \Phi \colon (\Delta_A,\Delta_A^+,\delta_A)\to (\Delta_B,\Delta_B^+,\delta_B)$, there is
a shift equivalence $(S,R)\colon A\sim_{SE} B$ of lag $l$ for some $l\ge 1$ and $m\ge 0$ such that $\Phi$ is the map described above (see \cite[Theorem 7.5.8]{L-M}).

Now we consider the special case given by a strong shift equivalence.

Let $A\in M_N(\Z^+)$ and $B\in M_M(\Z^+)$ be two non-negative matrices.
An {\it elementary shift equivalence} $(S,R): A\approx B$ from $A$ to $B$ is a pair of non-negative integral matrices $R$ and $S$ such that
$A=RS$ and $B=SR$. Note that this is just a shift equivalence of lag $l=1$.  A {\it strong shift equivalence of lag $l$}
from $A$ to $B$ is a sequence of $l$ elementary shift equivalences
\begin{equation}
\label{eq:sse}
(S_1,R_1)\colon A=A_0\approx A_1, (S_2,R_2) \colon A_1\approx A_2, \dots , (S_l,R_l)\colon A_{l-1} \approx A_l =B,
\end{equation}
see \cite[Definition 7.2.1]{L-M}.
Note that each strong shift equivalence as above induces a shift equivalence of lag $l$ from $A$ to $B$, and thus
an isomorphism $\Phi$ from the
dimension triple $(\Delta_A, \Delta _A^+, \delta_A)$
associated to $A$ to the dimension triple $(\Delta_B, \Delta_B^+, \delta_B)$ associated to $B$ \cite[7.5]{L-M},  such that
$$ \Phi (\iota_{n,\infty} (x)) = \iota_{m+n,\infty} ( S_l\cdots S_1x )$$
for $x\in \Z^N$, where $m$ is a fixed non-negative integer. By Theorem \ref{Theorem: MainFirstSection}, $\Phi$ induces an isomorphism,
also denoted by $\Phi$, from $K_0^{gr}(L(E))$ onto $K_0^{gr}(L(F))$.
Letting $S= S_l\cdots S_1$, $R= R_l\cdots R_1$,  we say that $(S,R)$ is a strong shift equivalence from $A$ to $B$. Two non-negative integral matrices $A$ and $B$ are
said to be {\it strongly shift equivalent} if there is a strong shift equivalence from $A$ to $B$.

In the situation above, we say that $\Phi$ is the morphism determined by $((S,R),m)$, or just by $(S,m)$. (Observe that $\Phi$ only depends on $S$ and $m$.)

Some examples can be useful to clarify these concepts.

\begin{exems}
 \label{exam:Moritamatrices}
 \begin{enumerate}
{\rm  \item Let $E$ be an essential finite graph, and consider the corner isomorphism $\alpha \colon L(E) \to pL(E)p$ built before.
  Then $\alpha$ induces a graded Morita-equivalence from $L(E)$ to itself. This corresponds to the elementary shift equivalence $(I_N, A)$,
  with $m=1=l$, where $A=A_E^t$. The inverse is given by the elementary shift equivalence $(A, I_N)$ with $m=0$ and $l=1$. It is realized by the
  graded Morita equivalence inverse to the one defined by $\alpha$.
  \item Let
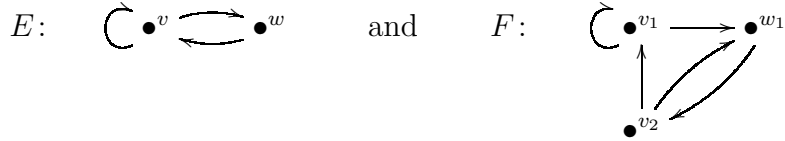
   $$E \colon \hspace{1truecm} \xymatrix{\bullet^{v} \ar@(dl,ul)
\ar@/^/[r] & \bullet^{w} \ar@/^/[l]} \qquad \text{ and  }\qquad F\colon \hspace{1truecm}\xymatrix{\bullet^{v_1} \ar@(dl,ul) \ar[r] & \bullet^{w_1}
\ar@/^/[dl] \\ \bullet^{v_2} \ar[u] \ar@/^/[ur]}$$ be the graphs considered in \cite[Example 1.10]{AbAnhLouP2}, where $F$ is obtained from $E$ by an in-splitting.
  Then the corresponding graded Morita equivalence induces at the level of $K_0^{gr}$ a strong shift equivalence $(S,R)$, with
  $S =\begin{pmatrix} 1 & 0 \\ 0 & 1 \\ 1 & 0 \end{pmatrix}$, $R= \begin{pmatrix}  1 & 1 & 0 \\ 0 & 0 & 1
  \end{pmatrix}$, with $m=l= 1$. It is important to remark here that, since $m=1$, the induced map $\Phi \colon K_0^{gr}(L(E))  \to K_0^{gr}(L(F))$
  does not preserve the canonical order-units. This corresponds to the fact that the process described in \cite{AbAnhLouP2} gives only a graded Morita equivalence
  in this case, but not an isomorphism. Considering the same pair $(S,R)$, but with $m=0$, $l=1$, one obtains another map
  $\Phi' \colon K_0^{gr}(L(E))  \to K_0^{gr}(L(F))$ which does preserve the order-unit. We will show later that this map can be lifted to a graded isomorphism
  from $L(E)$ to $L(F)$.
\item We consider now the graphs
$$E \colon  
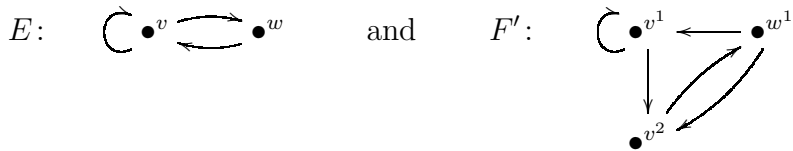
\hspace{1truecm} \xymatrix{\bullet^{v} \ar@(dl,ul)
\ar@/^/[r] & \bullet^{w} \ar@/^/[l]} \qquad \text{ and } \qquad  F' \colon  \hspace{1truecm} \xymatrix{\bullet^{v^1} \ar[d] \ar@(dl,ul) & \bullet^{w^1}
\ar@/^/[dl] \ar[l] \\ \bullet^{v^2} \ar@/^/[ur]}$$ arising in \cite[Example 1.13]{AbAnhLouP2}. The graph $E$ is the same as the one considered in the previous example, but
now the graph $F'$ is obtained from $E$ by an out-splitting instead of an in-splitting. In this case the corresponding graded Morita equivalence is indeed a graded isomorphism.
At the level of $K_0^{gr}$, it corresponds to the strong shift equivalence $(S,R)$, with $S= \begin{pmatrix} 1 & 0 \\ 1 & 0 \\ 0 & 1 \end{pmatrix}$,
$R=\begin{pmatrix} 1 & 0 & 1  \\ 0 & 1  & 0  \end{pmatrix}$, with $m=0$, $l=1$.}
   \end{enumerate}
\end{exems}

The first part of the following result has been proved by Hazrat \cite[Proposition 15(2)]{Hazrat2}.

\begin{theor}
\label{thm:gr-Morita}
Let $E$ and $F$ be finite essential graphs, and set $A=A_E^t$, $B=A_F^t$. If $A$ and $B$ are strongly shift equivalent then
$L(E)$ and $L(F)$ are graded Morita equivalent. Moreover, if $(S,R)$ is a strong shift equivalence from $A$ to $B$ and $m$ is any non-negative integer,
then there is a graded Morita equivalence $\Psi \colon \mbox{Gr}-L(E)\longrightarrow \mbox{Gr}-L(F)$ such that $K_0^{gr}(\Psi)$ is the isomorphism induced by
$((S,R),m)$.
 \end{theor}

\begin{proof}
Let $(S,R)$ be a strong shift equivalence from $A$ to $B$ and let $m$ be a non-negative integer. By \cite{W}, there is a finite
set of finite essential graphs $E_0,E_1,\dots , E_n$,
 with $E_0=E$ and $E_n=F$, such that each $E_i$ is obtained from $E_{i-1}$ by an out-splitting, or an in-splitting, or the
 inverse of one of these operations. (See also \cite[Chapters 2 and 7]{L-M}). Now there are matrices $S_i,R_i$ corresponding to these
 operations, so that $A_{E_{i-1}}^t=R_iS_i$ and $A_{E_i}^t= S_iR_i$ for $i=1,\dots ,n$ and $S=S_n\cdots S_1$ and $R=R_n\cdots R_1$.
  By \cite[Propositions 1.11 and 1.14]{AbAnhLouP2} each of these operations induces a
graded Morita equivalence
$$\Psi _i\colon \mbox{Gr}-L(E_{i-1})\longrightarrow \mbox{Gr}-L(E_i)$$
such that the induced map $K_0^{gr}(\Psi_i)$ is the map corresponding to $((S_i,R_i), m_i)$ for some $m_i\ge 0$ (where we identify $(K_0^{gr}(L(E_i)),
K_0^{gr}(L(E_i))^+, \mathcal T _*) $
with the corresponding dimension triple $(\Delta _{A_{E_i}^t}, \Delta _{A_{E_i}^t}^+, \delta_{A_{E_i}^t}))$.
Now $\Psi' := \Psi _n\circ \cdots \circ \Psi _1$ is a graded Morita equivalence from $\mbox{Gr}-L(E)$ to $\mbox{Gr}-L(F)$ and
$$ K_0^{gr}(\Psi) = K_0^{gr}(\Psi _n)\cdots K_0^{gr}(\Psi _1) = (S_n\cdots S_1, m_1+\cdots +m_n)= (S,m'),$$
where $m'=m_1+\cdots + m_n$.

Now if $m'=m$ we have reached the desired map $((S,R),m)$. If $m'<m$, consider
 $$((S_{n+1}, R_{n+1}), m-m') = ((I_M, B^{m-m'}), m-m'),$$
 with lag $m-m'$, see Example \ref{exam:Moritamatrices}(1). Set $m_{n+1}:= m-m'$. Then the composition of all $(S_i,m_i)$,
$i=1,\dots ,n+1$,  is precisely the map from the dimension triple of $A$ to the dimension triple of $B$ given by $(S, m)$. Similarly, if $m'>m$, consider
$$((S_{n+1}, R_{n+1}), m'-m)= ((B^{m'-m}, I_M),0),$$
with lag $m'-m$, see Example \ref{exam:Moritamatrices}(1). Then we again get that the composition of all $(S_i,m_i)$ (with $m_{n+1}:= m'-m$) is precisely $(S, m)$.

In any case, by Example \ref{exam:Moritamatrices}(1), there is a graded Morita equivalence $\Psi_{n+1}\colon \mbox{Gr}-L(F)\to \mbox{Gr}-L(F)$ such that
$K_0^{gr}(\Psi_{n+1})= (S_{n+1}, m_{n+1})$. Now set
$$\Psi: = \Psi_{n+1}\circ  \Psi '\colon \mbox{Gr}-L(E)\longrightarrow  \mbox{Gr}-L(F)$$
be the composition, which is a graded Morita equivalence.
(Of course $\Psi_{n+1}=\text{Id}$ if $m=m'$.) Then clearly $K_0^{gr}(\Psi)$ is the map induced by $((S,R),m)$.
 \end{proof}

Now we can make the following definition.

\begin{defi}
 \label{def:ssequeiv}
{\rm  Let $E$ and $F$ be finite essential graphs, and set $A=A_E^t$, $B=A_F^t$. We say that $L(E)$ and $L(F)$
 are {\it strongly shift $K_0^{gr}$-isomorphic} if there exists a strong shift equivalence $(S,R)$ from $A$ to $B$ and a non-negative integer $m$ such that
 the induced isomorphism $\Phi \colon K_0^{gr}(L(E))\to K_0^{gr}(L(F))$ satisfies that $\Phi ([1_{L(E)}])=[1_{L(F)}]$.}
 \end{defi}

In this case we can deduce graded isomorphism using  \cite{W}. First, we need a lemma.

\begin{lem}
 \label{lem:cancellationinK0gr}
 Let $E$ be a finite graph without sinks. Assume that $P,Q$ are graded finitely generated projective right $L(E)$-modules such that
 $[P]= [Q]$ in $K_0^{gr}(L(E))$. Then $P$ and $Q$ are graded-isomorphic.
\end{lem}

\begin{proof}
 Observe that $L(E)$ is strongly graded, so that $\mbox{Gr}-L(E)$ is equivalent to $\mbox{Mod}-L(E)_0$ by Dade's Theorem.
 It follows that there are finitely generated projective right $L(E)_0$-modules $P_0$ and $Q_0$ such that $P\cong_{gr} P_0\otimes _{L(E)_0}L(E)$
 and $Q\cong_{gr} Q_0\otimes _{L(E)_0} L(E)$, and $[P_0]=[Q_0]$ in $K_0(L(E)_0)$. Since $L(E)_0$ is an ultramatricial algebra, it has stable rank one,
 and so $P_0\cong Q_0$. It follows that $P\cong_{gr} Q$.
 \end{proof}

\begin{theor}
 \label{thm:l=1}
 Let $E$ and $F$ be finite essential graphs. If $L(E)$ and $L(F)$ are strongly shift $K_0^{gr}$-isomorphic then $L(E)$ and
 $L(F)$ are graded isomorphic. Moreover, let $\Phi \colon K_0^{gr}(L(E))\to K_0^{gr}(L(F))$ be an isomorphism satisfying the conditions in Definition
 \ref{def:ssequeiv}. Then there is a graded isomorphism
 $\phi\colon L(E)\to L(F)$ such that $K_0^{gr}(\phi) = \Phi$.
\end{theor}

\begin{proof}
By definition there is a strong shift equivalence $(S,R)$ from $A:=A_E^t$ to $B:=A_F^t$ and a non-negative integer $m$ such
that the induced map $K_0^{gr}(L(E))\to K_0^{gr}(L(F))$ is $\Phi$.

By Theorem \ref{thm:gr-Morita}, there is a graded Morita equivalence
$$\Psi\colon \mbox{Gr}-L(E) \longrightarrow \mbox{Gr}-L(F)$$
such that $K_0^{gr}(\Psi )=\Phi$.
We have
$$[1_{L(F)}] =\Phi ([1_{L(E)}])=K_0^{gr}(\Psi)([1_{L(E)}]) .$$
Set $P:= \Psi (L(E)_{L(E)})$. Then we get from the above that $[P]= [L(F)_{L(F)}]$ in $K_0^{gr}(L(F))$, so that, by Lemma \ref{lem:cancellationinK0gr},
we get that $P\cong_{gr}L(F)_{L(F)}$. By Example \ref{exam:grMoritaequi}(2), there is a graded isomorphism $\phi  \colon L(E)\to L(F)$ such that
$K_0^{gr}(\phi)= K_0^{gr}(\Psi)$. Since $K_0^{gr}(\Psi ) = \Phi$ we obtain that $K_0^{gr}(\phi) = \Phi$, as desired.
\end{proof}

\begin{exem}
 \label{exam:again-aureus}
{\rm  Let $E$ and $F$ be the graphs considered in Example \ref{exam:Moritamatrices}(2). As stated there the map induced by $((S,R),0)$ gives a strong shift equivalence
 $\Phi$ from the dimension triple of $A$ to the dimension triple of $B$ preserving the canonical order-units. By Theorem \ref{thm:l=1}, there is a graded isomorphism
 $\phi \colon L(E)\to L(F)$ such that $K_0^{gr}(\phi)=\Phi$. Observe that, following the proof of this theorem, we need to compose the graded Morita equivalence
 given in the proof of \cite[Proposition 1.11]{AbAnhLouP2} with the inverse of the graded Morita equivalence induced by the corner isomorphism
 $\beta \colon L(F)\to pL(F)p$ corresponding to $L(F)$, in order to get a graded Morita equivalence taking $[1_{L(E)}]$ to $[1_{L(F)}]$.}
 \end{exem}

\begin{qtn}
  {\rm Let $E$ and $F$ be finite essential graphs. Assume there is an isomorphism
  $$(K_0^{gr}(L(E)), K_0^{gr}(L(E))^+, [1_{L(E)}])\cong (K_0^{gr}(L(F)), K_0^{gr}(L(F))^+, [1_{L(F)}]).$$
  Write $A:=A_E^t$ and $B=A_F^t$. Assume in addition that $A$ and $B$ are strongly shift equivalent.
  Is then $L(E)$ graded isomorphic to $L(F)$? A positive answer would follow from Theorem \ref{thm:l=1}
  if one knows that every automorphism of the dimension triple $(\Delta_B, \Delta_B^+, \delta_B)$ is induced by
  a strong shift equivalence from $B$ to $B$.}
\end{qtn}

\section{Towards a characterization of graded isomorphisms}\label{MainSection}

The main goal of this section is to establish a strategy to prove a weak version of \cite[Conjecture 1]{Hazrat}
in the case of finite essential graphs.

In order to state our main result, we need some definitions.  Recall that an automorphism $g $ of an algebra $\mathcal A$ is said to be {\it locally inner}
in case, given any finite number of elements $x_1, \dots , x_n$ in $\mathcal A$ there is an inner automorphism $h$ of $\mathcal A$
such that $h(x_i)= \varphi (x_i)$ for $i=1, \dots , n$. It is well-known (\cite[Lemma 15.23(b)]{vnrr}) that any automorphism $\varphi $ of an ultramatricial algebra $\mathcal A$,
such that $K_0(\varphi ) = \text{id}$, is locally inner.

Let $E$ be a finite essential graph, and write $L(E)\cong L(E)_0[t_+,t_-,\alpha]$ as in Section \ref{sect:Equival}.
Let $g$ be a locally inner automorphism of $L(E)_0$. Then we define $L^g(E)$ by
$$L^g(E)= L(E)_0[t_+,t_-, g\circ \alpha ].$$
Observe that $L^g(E)$ is a graded algebra with the same graded $K$-theory as $L(E)$.

\begin{theor}\label{theor:Kiffgr-iso unital}
If $E,F$ are finite essential graphs, then the following are equivalent:
\begin{enumerate}
\item $(K_0(L(E)_0), K_0^+(L(E)_0), [1_{L(E)_0}])\cong_{\Z[x,x^{-1}]} (K_0(L(F)_0), K_0^+(L(F)_0), [1_{L(F)_0}])$.
\item  There exists a locally inner automorphism $g$ of $L(E)_0$ such that $L^g(E)\cong_{gr}L(F)$.
\end{enumerate}
\end{theor}
\begin{proof}
$(2)\Rightarrow (1)$. If $L^g(E)\cong_{gr}L(F)$, there is an
order-unit preserving  ordered  $\Z[x,x^{-1}]$-module isomorphism
$K_0^{gr}(L^g(E))\cong K_0^{gr}(L(F))$. So the result follows from
Theorem \ref{Theorem: MainFirstSection} and the fact that $K_0^{gr}(L^g(E))= K_0^{gr}(L(E))$.

$(1)\Rightarrow (2)$ will be proven in Section
\ref{ProofMainTheorem}.
\end{proof}

The proof of $(1)\Rightarrow (2)$ is more involved, and we will postpone it until next section. We will explain in the present section
the strategy we will follow to prove it, and we will give the final argument of the proof, assuming the results obtained
in the next section.\vspace{.2truecm}

We do not know whether the existence of a graded isomorphism between $L^g(E)$ and $L(F)$ implies the existence of a graded isomorphism
between $L(E)$ and $L(F)$.

\begin{exem}
{\rm A particular case to apply this argument is the problem of isomorphism
 between
$L_2$ and $L_{2_-}$ (see e.g. \cite[Section 2]{AbAnhLouP2}). If $L_2$ is graded %Morita-equivalent
isomorphic to $L_{2_-}$,
then by Theorem \ref{theor:Kiffgr-iso unital} $(2)\Rightarrow (1)$ and
Theorem \ref{Theorem: MainFirstSection} the corresponding matrices $A$ and $B$ are shift equivalent, and so
flow equivalent. By \cite{Franks} we then have $\mbox{det}(I-A_2)=\mbox{det}(I-A_{2_-})$, but
$\mbox{det}(I-A_2)=1$ and $\mbox{det}(I-A_{2_-})=-1$  (by direct calculation).
Thus, $L_2$ and $L_{2_-}$ are not graded %Morita-equivalent
isomorphic.
It is still open to decide the existence of a (nongraded) isomorphism.}
\end{exem}

We now begin to explain our strategy to show Theorem \ref{theor:Kiffgr-iso unital}.

\begin{noname}\label{startingdata}
{\rm \textbf{\textsc{Starting Data}:} Let $E,F$ be finite essential graphs,
such that $N=\vert E^0\vert ,
M=\vert F^0\vert$. We fix pictures $L(E)=L(E)_0[t_+, t_{-}; \alpha]$
and $L(F)=L(F)_0[s_+, s_{-}; \beta]$. Also, we define $A:=A_E^t\in
M_N(\Z^+)$ and $B:=A_F^t\in M_{M}(\Z^+)$, and fix $\delta_A:
K_0(L(E)_0)\rightarrow K_0(L(E)_0)$ and $\delta_B:
K_0(L(F)_0)\rightarrow K_0(L(F)_0)$ to be the group isomorphisms given
by multiplication by $A$ (respectively $B$).  If we write
$\mathcal M (E):=L(E)_0=\varinjlim (\mathcal M (E)_n, j_{n,m}^E)$, where $\mathcal M (E)_n = L(E)_{0,n}$, and
similarly $\mathcal M (F) : = L(F)_0=\varinjlim
(\mathcal M (F)_n, j_{n,m}^F)$ (whose connecting maps are induced by the
(CK2) relation on $E$ and $F$ respectively), then we denote $G_m:=
K_0(\mathcal M(E)_m)$, $H_m:= K_0(\mathcal M (F)_m)$, $G:=K_0(\mathcal M (E))$ and
$H:=K_0(\mathcal M (F))$. Notice that for every $m\geq 0$ we have $G_m\cong
\Z^N$ and $H_m\cong \Z^{M}$, as well as that $G=\varinjlim (G_m, A)$
and $H=\varinjlim (H_m, B)$.}
\end{noname}

In order to attain our goal, we start by establishing some previous
facts. The proof of the next two lemmas is trivial.

\begin{lem}\label{Lemma:CommCubeS}
If there exists a nonnegative integer matrix $S\in M_{M,N}(\Z^+)$ such that $SA=BS$, then for every $l,k, m\geq 1$ there exists a commutative diagram
$$\xymatrix@R=.25pc @C=.25pc{
%@!=.2pc{
%
& & & G_{l+1} \ar[rrrr]^{A^k}  \ar'[dd][ddddd]_S& & & & G_{k+l+1} \ar[ddddd]_S\\
& & & & & & & \\
G_{l} \ar[uurrr]_{id_{\Z^N}}\ar[rrrr]^{A^k} \ar[ddddd]_S& & & &G_{k+l} \ar[uurrr]_{id_{\Z^N}}\ar[ddddd]_S\\
& & & & & & &\\
& & & & & & &\\
& & & H_{m+1}  \ar'[r][rrrr]^{B^k}  & & & &H_{ m+k+1} \\
& & & & & & &\\
H_{m} \ar[uurrr]_{id_{\Z^{M}}}\ar[rrrr]_{B^k}   & & & & H_{m+k}\ar[uurrr]_{id_{\Z^{M}}}
}
$$
\end{lem}

\begin{lem}\label{Lemma:CommCubeR}
If there exists a nonnegative integer matrix $R\in M_{N,M}(\Z^+)$ such that $AR=RB$, then for every $l,k, m\geq 1$ there exists a commutative diagram
$$\xymatrix@R=.25pc @C=.25pc{
%@!=.2pc{
%
& & & G_{l+1} \ar[rrrr]^{A^k}  & & & & G_{k+l+1} \\
& & & & & & & \\
G_{l} \ar[uurrr]_{id_{\Z^N}}\ar[rrrr]^{A^k} & & & &G_{k+l} \ar[uurrr]_{id_{\Z^N}}\\
& & & & & & &\\
& & & & & & &\\
& & & H_{m+1} \ar'[uuu][uuuuu]_R \ar'[r][rrrr]^{B^k}  & & & &H_{ m+k+1} \ar[uuuuu]_R \\
& & & & & & &\\
H_{m} \ar[uurrr]_{id_{\Z^{M}}}\ar[uuuuu]_R \ar[rrrr]_{B^k}   & & & & H_{m+k}\ar[uurrr]_{id_{\Z^{M}}}\ar[uuuuu]_R
}
$$
$\mbox{ }$ \hspace{\fill}$\Box$
\end{lem}

Given $t\geq 0$, we will denote $\Omega_t:={id_{\Z^N}}: G_{t}\rightarrow G_{t+1}$ and $\Omega'_t:={id_{\Z^{M}}}: H_{t}\rightarrow H_{t+1}$. If we have
$$(K_0(\mathcal M (E)), K_0^+(\mathcal M (E)), [1_{\mathcal M (E)}])\cong_{\Z[x,x^{-1}]} (K_0(\mathcal M (F)), K_0^+(\mathcal M (F)), [1_{\mathcal M (F)}]),$$
then by Theorem \ref{Theorem: MainFirstSection}
$$(K_0(\mathcal M (E)), K_0^+(\mathcal M (E)), [1_{\mathcal M (E)}], \delta_A)\cong (K_0(\mathcal M (F)), K_0^+(\mathcal M (F)), [1_{\mathcal M(F)}], \delta_B)$$
through an ordered group isomorphism $\Phi : K_0(\mathcal M (E))\rightarrow
K_0(\mathcal M (F))$ such that $\Phi ([1_{\mathcal M (E)}])=[1_{\mathcal M (F)}]$.
Equivalently, by Theorem \ref{Theorem: MainFirstSection},
$A\sim_{SE} B$ with lag $l$ for some $l \geq 1$. Hence, there exists
a pair of nonnegative integral matrices $R\in M_{N\times M}(\Z^+)$
and $S\in M_{M\times N}(\Z^+)$ satisfying the four shift equivalence
equations:
\begin{enumerate}
\item (i) $AR=RB$; (ii) $SA=BS$.
\item (i) $A^{l}=RS$; (ii) $B^{l}=SR$.
\end{enumerate}

Notice that, when $l=1$, we have strong shift equivalence of $A_E$ and $A_F$, therefore $L(E)$ and $L(F)$ are strongly shift $K_0^{gr}$-isomorphic, and so by
Theorem \ref{thm:l=1}, we have that $L(E)\cong _{gr} L(F)$.

Consequently, we shall assume in the sequel that $l>1$.  Indeed, $(R, SA^{t} )$ implements
a shift equivalence between $A$ and $B$ of lag $l+t$ for any $t\ge 1$ (see \cite[[Proposition 7.3.2(2)]{L-M}).

\begin{rema}
\label{rem:R-Sexplicitisomorphisms}
{\rm By the proof of \cite[Theorems 7.5.7 \& 7.5.8]{L-M}, the maps $S$ and $R$ define explicitly the ordered group isomorphisms $\Phi$ and $\Phi^{-1}$ respectively.
In particular, $S$ and $R$ preserve the order unit (whenever $\Phi $ does).}
\end{rema}

Now, because of Theorem \ref{Theorem: MainFirstSection} we have
$$(K_0(\mathcal M (E)), K_0^+(\mathcal M(E)),  [1_{\mathcal M (E)}], \alpha_{\ast})\cong (K_0(\mathcal M (F)), K_0^+(\mathcal M (F)), [1_{\mathcal M (F)}], \beta_{\ast})$$
for the respective corner isomorphisms $\alpha$ and $\beta$, through the same isomorphism of ordered groups $\Phi$. Thus, by
Lemmas \ref{Lemma:CommCubeS} and \ref{Lemma:CommCubeR} we have, for
a suitable $m\geq 1$, the following commutative diagram (tagged
$(\dag)$)
$$\xymatrix @R=.75pc @C=.5pc{
& & & G_{1} \ar[rrrr]^{A^l}  \ar'[dd]'[dddd][ddddd]_(-.75)S& & & & G_{l+1}\ar[rrrr]^{A^l} \ar'[dd]'[dddd][ddddd]_(-.75)S& & & & G_{2l+1} \ar'[dd][ddddd]_(.35)S
\ar @{.>}[rrrr] & & & & G \ar[ddddd]_{\Phi}\\
& & & & & & & & & & & & & & &\\
G_{0} \ar[uurrr]^{\Omega_0}\ar[rrrr]^{A^l} \ar[ddddd]_(.4)S& & & &G_{l} \ar[uurrr]^{\Omega_l}\ar[rrrr]^{A^l}\ar[ddddd]_(.4)S& & & & G_{2l}
\ar[uurrr]^{\Omega_{2l}}\ar[ddddd]_(.4)S \ar @{.>}[rrrr] & & & & G \ar[uurrr]_{\alpha _*} \ar[ddddd]_{\Phi}\\
& & & & & & & & & & & & & & &\\
& & & & & & & & & & & & & & &\\
& & & H_{m+1} \ar'[uur]'[uuurr][uuuuurrrr]_(-.2)R \ar'[r][rrrr]_(.1){B^l}  & & & &H_{ m+l+1}\ar'[uur]'[uuurr][uuuuurrrr]_(-.2)R \ar'[r][rrrr]_(.1){B^l}
& & & &  H_{m+2l+1} \ar @{.>}'[r][rrrr] & & & & H\\
& & & & & & & & & & & & & & &\\
H_{m}\ar[uuuuurrrr]^R \ar[uurrr]_{\Omega'_m}\ar[rrrr]_{B^l}    & & & & H_{m+l} \ar[uuuuurrrr]^R \ar[uurrr]_{\Omega'_{m+l}}\ar[rrrr]_{B^l}
& & & & H_{m+2l}\ar[uurrr]_{\Omega'_{m+2l}} \ar @{.>}[rrrr] & & & & H \ar[uurrr]_{\beta _*}
}
$$
All the group morphisms appearing in this diagram are order preserving,
and, by Remark \ref{rem:R-Sexplicitisomorphisms}, those defined by matrices $R,S,A^l, B^l$ also preserve
order-unit, while the maps $\Omega_r, \Omega'_t$ are the restriction
of $\alpha_{\ast}$ (respectively $\beta_{\ast}$) to the
corresponding $G_r$ (respectively $H_t$) by the argument in the
proof of Lemma \ref{Lemma: ThirdEquiv}.

Our aim is to lift, in an inductive way, this diagram to a commutative diagram of algebras (tagged $(\diamondsuit)$)
$$\xymatrix @R=.3pc @C=.02pc{
& & & \mathcal M(E)_1 \ar[rrrr]^{j_{1,l+1}^E}  \ar'[dd]'[dddd][ddddd]_(-.75){\phi_{1}}& & & & \mathcal M (E)_{l+1}\ar[rrrr]^{j_{l+1, 2l+1}^E} \ar'[dd]'[dddd][ddddd]_(-.75){\phi_{l+1}}& & & &
\mathcal M (E)_{2l+1} \ar'[dd][ddddd]_(.35){\phi_{2l+1}} \ar @{.>}[rrrr] & & & & \mathcal M(E) \ar @/_/[ddddd]_{\varphi_1}\\
& & & & & & & & & & & & & & &\\
\mathcal M (E)_{0} \ar[uurrr]^{\alpha_{\vert \mathcal M (E)_{0}}}\ar[rrrr]^{j_{0,l}^E} \ar[ddddd]_(.4){\phi_{0}}& & & &\mathcal M (E)_{l} \ar[uurrr]^{\alpha_{\vert \mathcal M(E)_{l}}}\ar[rrrr]^{j_{l, 2l}^E}\ar[ddddd]_(.4){\phi_{l}}& & & &
\mathcal M (E)_{2l} \ar[uurrr]^{\alpha_{\vert \mathcal M (E)_{2l}}}\ar[ddddd]_(.4){\phi_{2l}} \ar @{.>}[rrrr] & & & & \mathcal M (E) \ar[uurrr]_{\alpha} \ar @/^/[ddddd]^{\varphi_0}\\
& & & & & & & & & & & & & & &\\
& & & & & & & & & & & & & & &\\
& & & \mathcal M (F)_{m+1} \ar'[uur]'[uuurr][uuuuurrrr]_(-.2){\psi_{m+1}} \ar'[r][rrrr]_(.1){j_{m+1, m+l+1}^F}  & & & &\mathcal M (F)_{m+l+1}
\ar'[uur]'[uuurr][uuuuurrrr]_(-.2){\psi_{m+l+1}} \ar'[r][rrrr]_(.1){j_{m+l+1, m+2l+1}^F}   & & & &  \mathcal M (F)_{m+2l+1} \ar @{.>}'[r][rrrr] & & & & \mathcal M (F) \ar @/_/[uuuuu]_{\psi_1}\\
& & & & & & & & & & & & & & &\\
\mathcal M (F)_{m}\ar[uuuuurrrr]^{\psi_{m}} \ar[uurrr]_(.8){\beta_{\vert \mathcal M (F)_{m}}}\ar[rrrr]_{j_{m, m+l}^F}    & & & &
\mathcal M (F)_{m+l} \ar[uuuuurrrr]^{\psi_{m+l}} \ar[uurrr]_(.8){\beta_{\vert \mathcal M (F)_{m+l}}}\ar[rrrr]_{j_{m+l, m+2l}^F}   & & & & \mathcal M (F)_{m+2l}
\ar[uurrr]_(.8){\beta_{\vert \mathcal M (F)_{m+2l}}} \ar @{.>}[rrrr] & & & & \mathcal M (F) \ar[uurrr]_{\beta}\ar @/^/[uuuuu]^{\psi_0}
}
$$

We will show how to do this in Section \ref{ProofMainTheorem}. For now, we assume that we have established that diagram $(\diamondsuit)$ commutes. Using this,
we are ready to finish the proof of Theorem \ref{theor:Kiffgr-iso unital}. Observe that, using the commutativity
of the front and the front-back diagrams,
we have well-defined
algebra isomorphisms $$\varphi _0, \varphi _1\colon \mathcal M (E) \to \mathcal M (F),$$ defined by
$$\varphi _{\varepsilon} (j^E_{kl+\varepsilon,\infty}(x)) = j^F_{kl+m+\varepsilon,\infty} (\phi _{kl+\varepsilon}(x))$$
for $x\in \mathcal M (E)_{kl+\varepsilon}$, $\varepsilon = 0,1$.

\smallskip

Similarly, we have well-defined algebra isomorphisms $\psi _{\varepsilon}\colon \mathcal M (F)\to \mathcal M (E)$, for $\varepsilon = 0,1$.
Observe that $\psi _{\varepsilon}$ is the inverse of $\varphi _{\varepsilon}$.

\smallskip

Notice that we have $\beta \cdot \varphi_0 = \varphi_1 \cdot \alpha$.

\begin{lem}
 \label{lem:locallyinner}
 Let $g=\psi _0 \varphi _1\in {\rm Aut}(\mathcal M (E))$. Then $g$ is a locally inner automorphism of $\mathcal M (E)$.
\end{lem}

\begin{proof}
 By the previous comment, it suffices to observe that $K_0(g)= \text{id}$. For this it is enough to show that, for any $k\in \N$, we have
 $$K_0(j^E_{kl+1,(k+2)l}) =  K_0 ( \psi _{m+(k+1)l}\cdot j^F_{m+kl+1, m+(k+1)l} \cdot \phi _{kl+1} ).$$

Note that
$$K_0(\psi_{m+(k+1)l}\cdot j^F_{m+k1+1,m+(k+1)l}\cdot  \phi_{kl+1})=  RB^{l-1}S= A^{l-1}RS=A^{2l-1}= K_0(j^E_{kl+1, (k+2)l}) \, ,$$
where we have used that $RB=AR$ and that $RS=A^l$. This shows the result.
 \end{proof}

 {\it Proof of Theorem \ref{theor:Kiffgr-iso unital}:}
 Set  $g= \psi _0 \varphi _1$. By Lemma \ref{lem:locallyinner}, $g$ is a locally inner automorphism of $\mathcal M (E)_0$.
 Put $\alpha ' = g\alpha $. In order to show that
$$L^g(E) =\mathcal M (E)[t_+,t_-, \alpha '] \cong \mathcal M (F)[s_+,s_-,\beta ]= L(F) \, ,$$
we first observe that  $\beta \cdot \varphi _0 = \varphi _0 \cdot \alpha'$. Indeed, we have
$$\varphi _0 \cdot \alpha'= \varphi _0 \cdot \varphi _0 ^{-1} \cdot \varphi _1\cdot \alpha = \varphi _1\cdot \alpha = \beta  \cdot \varphi_0.$$

Since $$s_+\varphi _0 (a) s_- = \beta (\varphi _ 0(a)) = \varphi _ 0 (\alpha ' (a))$$
for all $a\in \mathcal M (E)$, the universal property of $\mathcal M (E)[t_+,t_{-}; \alpha ']$ gives that there is a unique algebra homomorphism
$$\varphi: \mathcal M (E)[t_+,t_{-}; \alpha '] \rightarrow \mathcal M (F)[s_+,s_{-};\beta] $$
such that  $\varphi|_{ \mathcal M (E)}=\varphi_0$,
$\varphi (t_+)=s_+$ and $\varphi(t_{-})=s_{-}$.
Clearly $\varphi $ is a graded homomorphism.

The inverse of $\varphi$ is the unique algebra homomorphism
$$\psi \colon \mathcal M (F)[s_+,s_{-}; \beta] \rightarrow \mathcal M (E)[t_+,t_{-}; \alpha] $$
such that $\psi|_{ \mathcal M (F)}=\psi_0$,
$\psi(s_+)=t_+$ and $\psi(s_{-})=t_{-}$.
\qed

\vspace{.2truecm}

\begin{rema}
 \label{rem:no-iso}
 {\rm   Clearly, we have $L(E)\cong _{gr} L^g(E)$ when $g$ is an inner automorphism of $\mathcal M (E)$.
 As we shall see now, one can easily show that there are locally inner automorphisms $g$ of $\mathcal M(E)$ such that $L(E)\ncong_{gr} L^g(E)$, and indeed such that
 $L^g(E)\ncong L(F)$ for any finite graph $F$.
 However the automorphism $g$ produced in the proof of Theorem \ref{theor:Kiffgr-iso unital} is of a special form, and we have not been able to determine
 whether $L(E)\cong L^g(E)$ in that case.

 Now let us give an example of a locally inner $g$ such that $L^g(E)$ is not finitely generated as a $K$-algebra. This of course implies that
 $L^g(E)\ncong L(F)$ for any finite graph $F$, because $L(F)$ is a finitely generated $K$-algebra.

 Let $K$ be an infinite field.

Let $A=\begin{pmatrix} 1 & 0 \\ 1 & 1 \end{pmatrix}$, and let $E$ be
 the graph with adjacency matrix $A^t$. So $E^0=\{v,w\}$, and $E^1=\{ e,f,g  \}$, with $s(e)=r(e)= s(f)= v$ and
 $r(f)= s(g)= r(g) = w$. We consider the elements $t_+= e+g$, $t_-= e^*+g^*$, and the corner isomorphism $\alpha $ given by
 $\alpha (x) = t_+xt_-$ for $x\in \mathcal M (E)$.

  Observe that $\mathcal M (E)_n=K\times M_{n+1}(K)$, with transition maps given by
 $$(\lambda , Z) \mapsto (\lambda , \begin{pmatrix} \lambda & 0  \\ 0 & Z \end{pmatrix}).$$
 With this picture, $\alpha $ is given by
  $$(\lambda , Z) \mapsto (\lambda , \begin{pmatrix} 0 & 0  \\ 0 & Z \end{pmatrix}).$$
 We can build a locally inner automorphism $g$ of $\mathcal M (E)$ as follows. Let $(\xi _n: n\ge 1)$ be
 a sequence of elements of $K^{\times}$. We require that $\xi_{n+1}$ does not belong to the multiplicative group generated by
 $\xi_1,\dots , \xi_n$ for all $n\ge 1$. (Note that such a sequence exists, since $K$ is an infinite field.)
  Then define $U_0= (1,1)$,  and inductively $U_n= (\xi_n, X_n)$, with $X_{n}= \text{diag}(\xi_n , X_{n-1})$
for $n\ge 1$, where $X_0=1$. Set
$$g(x) = U_nxU_n^{-1},\qquad x\in \mathcal M (E)_n .$$
Then $g$ is a locally inner automorphism of $\mathcal M (E)$. Observe that $\alpha \circ g = g\circ \alpha$.

Set $L^g(E) = \mathcal M (E) [s_+,s_-,g\circ \alpha ]$. We show that $L^g(E)$ is not finitely generated as a $K$-algebra.
Assume that $L^g(E)$ is finitely generated. Then there is $n_0\ge 1$ such that $L^g(E)$ is generated by
$\mathcal M (E)_{n_0}$ and $s_+,s_-$. Set  $f^{(n)}= (1,0)\in  \mathcal M (E)_n$, and denote by $e^{(n)}_{ij}$,  $1\le i,j\le n+1$
the matrix units of the factor $M_{n+1}(K)$ of $\mathcal M (E)_n$. Set
$$\mathcal B = \{f^{(n_0)}\}\cup \{ e^{(n_0)}_{ij}: 1\le i,j\le n+1 \}, $$
which is a linear basis for $\mathcal M (E)_{n_0}$. Then every element $x$ in $\mathcal M (E)$ can be written as a linear combination of the
family
$$ \mathcal B ' = \{ b_rs_+b_{r-1}s_+ \cdots s_+b_0s_-b_{-1}s_-\cdots s_-b_{-r}: \,\, b_{-r},\dots ,b_0,\dots , b_r\in \mathcal B  \}.$$
(This uses the fact that $s_-bs_+= (g\alpha )^{-1} (s_+s_-bs_+s_-) \in  \mathcal M (E)_{n_0}$ for all $b\in \mathcal B$.)
Hence we can write
$$e^{(n_0+1)}_{1, n_0+2} = \sum _{b'\in \mathcal F} \lambda _{b'} b',$$
where $\mathcal F$ is a finite linearly independent subset of $\mathcal B '$, and $\lambda_{b'}\in K^{\times}$.
Observe that for $$b'=  b_rs_+b_{r-1}s_+ \cdots s_+b_0s_-b_{-1}s_-\cdots s_-b_{-r}\in \mathcal B ',$$
we have
\begin{align*}
 g(b') & = g(b_rs_+b_{r-1}s_+ \cdots s_+b_0s_-b_{-1}s_-\cdots s_-b_{-r}) \\
 &= g\Big( b_r(s_+b_{r-1}s_-)(s_+^2b_{r-2}s_-^2)\cdots (s_+^rb_0s_-^r)\cdots (s_+b_{-r+1}s_-)b_{-r}\Big)\\
 & = g\Big( b_r (g\alpha(b_{r-1}))(g\alpha)^2(b_{r-2}) \cdots (g\alpha)^r(b_0)\cdots (g\alpha)(b_{-r+1}) b_{-r}\Big) \\
 & = g(b_r) (g\alpha(g(b_{r-1})))(g\alpha)^2(g(b_{r-2})) \cdots (g\alpha)^r(g(b_0))\cdots (g\alpha)(g(b_{-r+1})) g(b_{-r})\Big) \\
 & = \zeta _{b'} b'
 \end{align*}
where $\zeta_{b'}\in \langle \xi_1,\dots ,\xi_{n_0}\rangle $, the subgroup of $K^{\times}$ generated by $\xi_1,\dots , \xi_{n_0}$.
Therefore we obtain
$$ \xi _{n_0+1} e^{(n_0+1)}_{1,n_0+2}= g(e^{(n_0+1)}_{1,n_0+2}) = \sum _{b'\in \mathcal F} \lambda _{b'} \zeta_{b'} b'.$$
Since $\mathcal F$ is a linearly independent family, we get $\xi_{n_0+1} =\zeta_{b'}$
for all $b'\in \mathcal F$. Since $\zeta_{b'}\in \langle \xi_1,\dots ,\xi_{n_0}\rangle $, we obtain a contradiction with
the choice of the family $(\xi_n)$. }
\end{rema}

 \section{Lifting the diagram}\label{ProofMainTheorem}

In this section we will proceed to prove Theorem \ref{theor:Kiffgr-iso unital}. According to the arguments
in Section \ref{MainSection}, it is enough to perform the correct lifting of the diagram $(\dag)$ to a diagram
$(\diamondsuit)$.\vspace{.2truecm}

The goal is to define algebra isomorphisms $\varphi _0, \varphi _1\colon \mathcal M (E) \to \mathcal M (F)$ by
$$\varphi _{\varepsilon} (j^E_{kl+\varepsilon,\infty}(x)) = j^F_{kl+m+\varepsilon,\infty} (\phi _{kl+\varepsilon}(x))$$
for $x\in \mathcal M (E)_{kl+\varepsilon}$, $\varepsilon = 0,1$ lifting the group diagram, so that $\beta\cdot \varphi_0=\alpha\cdot \varphi_1$.
In order to do that in a clear way, we will divide the proof in several intermediate steps.\vspace{.2truecm}

\begin{nota}\label{Notation:BigThm}
{\rm We start by fixing some notation that will be useful in the sequel. We will denote $E^0=\{v_1, \dots,v_N\}$ and $F^0=\{w_1,
\dots,w_M\}$. Also, for $d\geq 1$, we will denote $E^dv_i=\{\gamma
\in E^d \mid r(\gamma)=v_i\}$ and $v_iE^d=\{\gamma \in
E^d\mid s(\gamma)=v_i\}$ (analogously for the $w_j$'s of $F^0$); also, we will denote $v_jE_dv_i:=v_jE^d\cap E^dv_i$. For
a path $\gamma $ in $E$ of length $t$ and $d\ge t$, we will denote
by $\gamma E^{d-t}$ the set of paths $\lambda \in E^d$ such that
the initial segment of $\lambda $ of length $t$ is equal to $\gamma
$. Of course, similar notations apply to $F$.

In order to define correctly the maps $\alpha, \beta$, we fix for
each $1\leq i\leq N$ an edge $e_i\in r^{-1}(v_i)$, and analogously we
fix for each $1\leq j\leq M$ an edge $f_j\in r^{-1}(w_j)$. (Note that we use our hypothesis that $E$ and $F$ do not have sources here.)
Then, we
define:\begin{enumerate}
\item $t_+:=\sum\limits_{i=1}^N e_i$ and $t_-:=\sum\limits_{i=1}^N {e_i}^*$.
\item $s_+:=\sum\limits_{j=1}^M f_j$ and $s_-:=\sum\limits_{j=1}^M {f_j}^*$.
\end{enumerate}
So, the map $\alpha$ is defined by the rule $\alpha (a)=t_+at_-$, while the map $\beta$ is defined by the rule $\beta (b)=s_+bs_-$.

For a path $\gamma $ in $E^d$ we denote by $\widehat{\gamma}$ the path $e_i\gamma$ in $E^{d+1}$, where $v_i=s(\gamma )$.
Similarly, for $\lambda \in w_jF^d$, we set $\widehat{\lambda}:=  f_j \lambda \in F^{d+1}$.
Note that, for $\gamma _1,\gamma _2\in E^d$, we have $\alpha (\gamma _1\gamma _2^*) = \widehat{\gamma_1}\widehat{\gamma_2}^*$, and similarly,
for $\lambda _1,\lambda _2\in F^d$, we have $\beta (\lambda _1\lambda _2^*)= \widehat{\lambda _1}\widehat{\lambda_2}^*$.
}
\end{nota}

The proof consists of constructing certain partitions of the sets of paths of $E$ and $F$,
and using them to define the liftings of $R$ and $S$ along all the diagram.
Once the first three commutative squares of the diagram have been lifted, the rest of the partitions and
the corresponding liftings are completely determined by the choices made so far, so in a sense there is only a
finite number of choices involved in the process. Accordingly, we will divide the section in four subsections, corresponding to the liftings
of the first three commutative diagrams and to the final argument showing how to continue the process inductively to obtain the lifting of all the
squares.

\subsection{Lifting the first (``left end'') commutative square}

We start with the commutative square

$$\xymatrix{\Z^N\ar[rr]^{\Omega_0} \ar[dd]_{S} & & \Z^N\ar[dd]^{S}\\
 & & \\
\Z^M\ar[rr]_{\Omega '_m} & & \Z^M}
$$

and we lift it to a commutative square

$$\xymatrix{\mathcal M (E)_0\ar[rr]^{\alpha}\ar[dd]_{\phi_0} & & \mathcal M (E)_1\ar[dd]^{\phi_1}\\
 & & \\
\mathcal M (F)_m\ar[rr]_{\beta} & & \mathcal M (F)_{m+1} }
$$
where $\phi_0, \phi_1$ are unital morphisms. Because of the previous data, $1\leq m\leq l$ are fixed across the whole process. Here $A=(a_{ik})$ and $B=(b_{jt})$ are the
transposes of the adjacency matrices for $E$ and $F$ respectively. The matrices $S$ and $R$ have components denoted by $s_{ji}$ and $r_{ij}$ respectively.
We will consistently use $i$ to denote an index in the range $\{ 1, \dots ,N \}$,  and $j$ to denote an index in the range
$\{ 1,\dots , M\}$.

Now, we will define $\phi _0$. Since $S$ is a group homomorphism preserving the order-unit, we have that $|F^mw_j|=\sum
_{i=1}^N s_{ji}$. So, we can consider a partition
$$F^mw_j=\bigsqcup _{i=1}^N \Lambda_{j}^{m, v_i}$$
where $|\Lambda _{j}^{m, v_i}|= s_{ji}$ for all $j,i$. We will write
$$\Lambda _{j}^{m, v_i}=\{\lambda_{j1}^0(v_i)\dots ,\lambda_{js_{ji}}^0(v_i)\}.$$
We define a map $\phi_0\colon \mathcal M (E)_0 \to \mathcal M (F)_m$ by the rule
$$\phi _0(v_i)= \sum _{j=1}^M \sum _{\lambda \in \Lambda _{j}^{m, v_i}} \lambda \lambda ^*.$$
Note that
$$[\phi _0(v_i)]= \sum_{j=1}^M s_{ji}[w_j]=S[v_i]\, ,$$
so that $\phi _0$ lifts $S$.

Also, since $1=\sum\limits_{i=1}^N v_i$ and
$F^m=\bigsqcup\limits_{j=1}^M F^mw_j$, we have
$$\phi_0(1)=\sum\limits_{i=1}^N \phi_0(v_i)=\sum\limits_{i=1}^N\sum _{j=1}^M \sum _{\lambda \in \Lambda_{j}^{m, v_i}} \lambda \lambda ^*=
\sum _{j=1}^M \sum _{\lambda \in F^mw_j} \lambda \lambda ^*=\sum
_{\lambda \in F^m} \lambda \lambda ^*=1$$ where the last equality
follows from (CK2) and the fact that $F$ does not have sinks.

\vspace{.2truecm}

Now, we proceed to define $\phi_1$.  For this, we consider $F^{m+1}w_j$. Observe that
$$|F^{m+1}w_j|= \sum _{k=1}^M b_{jk}|F^mw_k|=
\sum _{k=1}^M \sum _{i=1}^N b_{jk}s_{ki}= \sum _{i=1}^N
(\sum_{k=1}^N s_{jk}a_{ki}) ,$$ where we have used that $SA=BS$. Notice that $a_{ki}$ is the number of edges in $v_iE^1v_k$. It
follows that we can consider a partition
\begin{equation} \label{eq:sqdecom1}
F^{m+1}w_j=\bigsqcup _{i=1}^N \bigsqcup_{k=1}^N \bigsqcup
_{\{e\in v_iE^1v_k \}} \Lambda _{j}^{m+1, e}\, ,
\end{equation}
with $|\Lambda _{j}^{m+1, e}|= s_{jk}$ for all $e\in v_iE^1v_k$. We will write
$$\Lambda _{j}^{m+1, e}=\{\lambda_{j1}^1(e), \dots , \lambda_{js_{jk}}^1(e)\}.$$
Moreover, we can choose the partition
(\ref{eq:sqdecom1}) in such a way that
\begin{equation} \label{eq:Ident1}
\Lambda _{j}^{m+1, \widehat{v_i}}=\Lambda_j^{m+1, e_i}= \{  \widehat{\lambda} \mid \lambda \in \Lambda_j^{m,v_i}\}.
\end{equation}
This makes sense because the family on the right hand
side is a family of $s_{ji}$ paths of length $m+1$ ending in $w_j$. In fact, we can fix as a condition
$$
\textbf{(\mbox{Ex }1)}\hspace{1 truecm} \lambda_{jt}^1(e_k)=\widehat{\lambda_{jt}^0(v_k)}
$$
for all the possible values of $j,k,t$. This is coherent with the possible values of these variables, and clearly Condition (Ex 1) implies (\ref{eq:Ident1}).

Now we define $\phi_1\colon \mathcal M (E)_1\to \mathcal M (F)_{m+1}$ as follows.
Given $e,f\in E^1$ such that $r(e)=v_k=r(f)$, we put:
$$\phi_1 (ef^*)= \sum _{j=1}^M \sum _{p=1}^{s_{jk}} \lambda_{j,p}^1(e)
\lambda _{j,p}^1(f)^*\, .$$

It is easy to show that $\phi_1$ is a well-defined unital homomorphism,
such that $K_0(\phi_1)=S$. Let us show that $\phi_1\alpha=\beta
\phi_0 $:
$$\phi _1(\alpha (v_i))=\phi _1(e_i e_i^*)= \sum _{j=1}^M \sum _{ \lambda \in \Lambda_j^{m+1, e_i}  } \lambda\lambda^*
= \sum _{j=1}^M \sum _{ \lambda\in \Lambda_j^{m,v_i} } \widehat{\lambda}\widehat{\lambda}^*= \beta (\phi_0(v_i)).
$$
where we have used (\ref{eq:Ident1}) for the third equality.

\vspace{.2truecm}

\subsection{Lifting the second (``diagonal'') commutative square}

We will take the commutative diagram

$$\xymatrix{\Z^N\ar[rr]^{\Omega_l}  & & \Z^N\\
 & & \\
\Z^M\ar[rr]_{\Omega'_m} \ar[uu]_{R}& & \Z^M\ar[uu]^{R}}
$$

and we will lift it to a commutative diagram

$$\xymatrix{\mathcal M (E)_{l}\ar[rr]^{\alpha} & & \mathcal M (E)_{l+1}\\
 & & \\
\mathcal M (F)_{m}\ar[rr]_{\beta} \ar[uu]_{\psi_m}& & \mathcal M (F)_{m+1} \ar[uu]^{\psi_{m+1}}}
$$
where $\psi_m, \psi_{m+1}$ are unital morphisms satisfying that $\psi_m\cdot \phi_0=j_{0,l}^E$ and $\psi_{m+1}\cdot \phi_1=j_{1,l+1}^E$.\vspace{.2truecm}

First, we will define $\psi_m$. Since  $\sum\limits_{j=1}^M
r_{ij}s_{jk}=(A^l)_{ik}$ is the number of paths in $v_kE^lv_i$ and $|\Lambda_j^{m,v_k}|= s_{jk}$, we can consider a
partition
\begin{equation}\label{eq:Part2}
v_kE^lv_i =\bigsqcup\limits _{j=1}^M \bigsqcup\limits _{\{\lambda \in \Lambda _j^{m, v_k}\}} \Gamma_{i}^{l,\lambda}
\end{equation}
with $\vert \Gamma_{i}^{l,\lambda}\vert= r_{ij}$ for all $\lambda\in \Lambda_j^{m,v_k}$.
Notice that (\ref{eq:Part2}) implies
\begin{equation}\label{eq:Part2(i)}
E^lv_i=\bigsqcup\limits _{k=1}^N \bigsqcup\limits _{j=1}^M \bigsqcup\limits _{\{\lambda \in \Lambda_j^{m, v_k}\}} \Gamma_{i}^{l,\lambda}
\end{equation}
and
\begin{equation}\label{eq:Part2(ii)}
v_kE^l=\bigsqcup\limits _{i=1}^N \bigsqcup\limits _{j=1}^M \bigsqcup\limits _{\{\lambda\in \Lambda_j^{m, v_k}\}} \Gamma_{i}^{l,\lambda}
\end{equation}

If $\lambda =\lambda_{jt}^0(v_k)\in F^mw_j$, we will write
$$\Gamma_i^{l,\lambda}=\Gamma_i^{l, \lambda_{jt}^0(v_k)}=\{\gamma_{i1}^m(\lambda_{jt}^0(v_k)), \dots ,\gamma_{ir_{ij}}^m(\lambda_{jt}^0(v_k))\}.$$
In particular, $s(\gamma_{ip}^m(\lambda_{jt}^0(v_k)))=v_k$ by (\ref{eq:Part2}).\vspace{.2truecm}

We define a map $\psi_m:\mathcal M (F)_{m}\rightarrow \mathcal M (E)_{l}$ by the rule
$$\psi_m(\lambda _1\lambda_2^*)=\sum\limits_{i=1}^N \sum\limits_{p=1}^{r_{ij}}\gamma_{ip}^m(\lambda_1){\gamma_{ip}^m(\lambda_2)}^*,$$
where $\lambda_1, \lambda_2\in F^mw_j$.\vspace{.2truecm}

As in the previous cases, it is easy to show that $\psi_m$ is a well-defined unital morphism satisfying that $K_0(\psi_m)=R$. Now,
\begin{equation}
\label{eq:News2}
\psi_m\cdot \phi_0(v_k)=\sum\limits_{j=1}^M \sum\limits_{\lambda\in \Lambda_j^{m,v_k}}\psi_m(\lambda \lambda^*)=
\sum\limits_{j=1}^M \sum\limits_{i=1}^N \sum\limits_{\lambda\in \Lambda_j^{m,v_k}} \sum\limits_{\gamma\in \Gamma_i^{l,\lambda}} \gamma \gamma^*
=\sum\limits_{\gamma\in v_kE^l}\gamma \gamma^*=j_{0,l}^E(v_k),
\end{equation}
where we have used (\ref{eq:Part2(ii)}) for the third equality.

Now, we will define $\psi_{m+1}$. Since, for $e\in E^1v_k$, we
have
$$|eE^lv_i|= |e(v_kE^lv_i)|= \sum _{j=1}^M r_{ij}s_{jk}\, , $$ we may consider a
partition
\begin{equation}\label{eq:Part3}
eE^lv_i =\bigsqcup\limits_{j=1}^M \bigsqcup\limits_{\{ \lambda\in \Lambda_j^{m+1, e}\}} \Gamma_i^{l+1, \lambda},
\end{equation}
with $\vert \Gamma_i^{l+1, \lambda}\vert =r_{ij}$ for $\lambda \in
\Lambda_j^{m+1, e}$. Given any $\lambda\in F^{m+1}
w_j$, we will write
$$\Gamma_i^{l+1, \lambda}=\{ \gamma_{i1}^{m+1}(\lambda), \dots , \gamma_{ir_{ij}}^{m+1}(\lambda) \}.$$
For $e\in E^1v_k$, left multiply (\ref{eq:Part2}) by $e$ to get
\begin{equation}\label{eq:News1}
eE^lv_i =e(v_kE^lv_i)= \bigsqcup\limits_{j=1}^M \bigsqcup\limits_{\{ \lambda\in \Lambda_j^{m, v_k}\}} e \Gamma_i^{l, \lambda}.
\end{equation}
Since there is a bijection $\lambda_{jt}^0 (v_k) \longleftrightarrow  \lambda^1_{jt}(e)$ between $\Lambda_j^{m,v_k}$ and $\Lambda _j^{m+1, e}$, we may
take as definition of the $\gamma _{ip}^{m+1}(\lambda _{jt}^1 (e))$ the following:
\begin{equation}\label{eq: Cond2}
\gamma_{ip}^{m+1}(\lambda_{jt}^1(e)):=e\gamma_{ip}^{m}(\lambda_{jt}^0(v_k)),
\end{equation}
for $e\in E^1v_k$, which is coherent with (\ref{eq:Part3}).

Now, we define a map $\psi_{m+1}:\mathcal M (F)_{m+1}\rightarrow \mathcal M (E)_{l+1}$ by the rule
$$\psi_{m+1}(\lambda _1\lambda_2^*)= \sum\limits_{i=1}^N \sum\limits_{p=1}^{r_{ij}}\gamma_{ip}^{m+1}(\lambda_1){\gamma_{ip}^{m+1}(\lambda_2)}^*.$$
for any $\lambda_1, \lambda_2 \in F^{m+1} w_j$
As in the previous cases, it is easy to show that $\psi_{m+1}$ is a well-defined unital morphism satisfying that $K_0(\psi_{m+1})=R$.

Given $e,f\in E^1 v_k$, we compute that
$$\psi_{m+1}\cdot \phi_1(ef^*)=\sum\limits_{j=1}^M \sum\limits_{p=1}^{s_{jk}}\psi_{m+1}(\lambda_{jp}^1(e){\lambda_{jp}^1(f)}^*)= e j^E_{0,l}(v_k)f^*= j^E_{1,l+1}(ef^*),$$
where we have used (\ref{eq: Cond2}) and (\ref{eq:News2}).

Now, using (\ref{eq: Cond2}) and (Ex 1) we compute
$$e_k \gamma_{ip}^m(\lambda_{jt}^0(v_k))= \gamma_{ip}^{m+1}(\lambda_{jt}^1(e_k)) =\gamma_{ip}^{m+1}(\widehat{\lambda_{jt}^0(v_k)}).$$

Summarizing, for any $\lambda \in \Lambda_j^{m, v_k}$ we proved that
$$ \textbf{(\mbox{Ex }2)}\hspace{1 truecm} \widehat{\gamma_{ip}^m(\lambda)}=\gamma_{ip}^{m+1}(\widehat{\lambda}).$$

Finally, given $\lambda_1 \in \Lambda_j^{m, v_{k_1}}$, $\lambda_2 \in \Lambda_j^{m, v_{k_2}}$, we compute that
$$\alpha\cdot \psi_m (\lambda_1{\lambda_2}^*)= \sum\limits_{i=1}^N \sum\limits_{p=1}^{r_{ij}} \widehat{\gamma_{ip}^m(\lambda_1)}\widehat{\gamma_{ip}^m(\lambda_2)}^* =
\sum\limits_{i=1}^N \sum\limits_{p=1}^{r_{ij}} \gamma_{ip}^{m+1}(\widehat{\lambda_1})\gamma_{ip}^{m+1}(\widehat{\lambda_2})^*= \psi_{m+1}\cdot \beta (\lambda_1\lambda_2^*),$$
showing that $\alpha  \cdot \psi_m= \psi_{m+1}\cdot \beta $, as desired.\vspace{.2truecm}

\subsection{Lifting the third (``second from the left end'') commutative square}

Next, we will perform the lifting of the third commutative square of
diagram $(\dag)$ to that of diagram $(\diamondsuit)$. We will take the
commutative diagram

$$\xymatrix{\Z^N\ar[rr]^{\Omega _l} \ar[dd]_{S} & & \Z^N\ar[dd]^{S}\\
 & & \\
\Z^M\ar[rr]_{\Omega '_{m+l}} & & \Z^M}
$$

and we will lift it to a commutative diagram

$$\xymatrix{\mathcal M (E)_{l}\ar[rr]^{\alpha}\ar[dd]_{\phi_l} & & \mathcal M (E)_{l+1}\ar[dd]^{\phi_{l+1}}\\
 & & \\
\mathcal M (F)_{m+l}\ar[rr]_{\beta} & & \mathcal M (F)_{m+l+1} }
$$
where $\phi_l, \phi_{l+1}$ are unital morphisms satisfying that $\phi_l\cdot\psi_m = j_{m,m+l}^F$ and $\phi_{l+1}\cdot\psi_{m+1} = j_{m+1,m+l+1}^F$.\vspace{.2truecm}

As in the previous case, we start by defining a suitable
partition for $F^{m+l}w_j$. For $\lambda\in F^m w_k$,  we can define a partition
\begin{equation}\label{eq:Part4}
\lambda F^{l} w_j=\bigsqcup\limits_{i=1}^N \bigsqcup\limits_{\{ \gamma\in \Gamma_i^{l,\lambda}\}} \Lambda_j^{m+l, \gamma}
\end{equation}
where $\vert \Lambda_j^{m+l, \gamma} \vert= s_{ji}$ for $\gamma \in
\Gamma _i^{l, \lambda }$. Given any
$\gamma\in \Gamma_i^{l,\lambda}$, we will
write
$$\Lambda_j^{m+l, \gamma}=\{ \lambda_{jq}^l(\gamma)\mid 1\leq q\leq s_{ji}\}.$$

Now, we will fix a condition analog to (\ref{eq: Cond2}). We observe that
for any $1\leq j,k\leq M$ we have $\sum\limits_{i=1}^N s_{ji}r_{ik}=(B^l)_{jk}$. Thus, we can consider an enumeration
\begin{equation}\label{Eq:OtraMas}
w_kF^lw_j=\{ \lambda_{ipq}^{jk}\mid 1\leq i\leq N, 1\leq p\leq r_{ik}, 1\leq q\leq s_{ji}\}.
\end{equation}
Since $r(\lambda^0_{kt}(v_d))=w_k$ for all $t,d$, we can fix as definition
\begin{equation}\label{eq:Cond3}
\lambda_{jq}^l(\gamma_{ip}^m(\lambda_{kt}^0(v_d))):=\lambda_{kt}^0(v_d)\cdot \lambda_{ipq}^{jk}.
\end{equation}

Now, we define $\phi_l: \mathcal M (E)_{l}\rightarrow \mathcal M (F)_{m+l}$ as follows: for any $\gamma_1, \gamma_2 \in E^l v_i$,
$$\phi_l(\gamma _1\gamma_2 ^*)=\sum\limits_{j=1}^M\sum\limits_{q=1}^{s_{ji}}\lambda_{jq}^l(\gamma_1)\lambda_{jq}^l(\gamma_2)^*.$$
It is easy to show that $\phi_l$ is a unital morphism satisfying
$K_0(\phi_l)=S$.

Given $\lambda_1,\lambda_2 \in F^mw_k$, we compute
$$\phi_l\cdot \psi_m (\lambda_1 \lambda_2^*)=\sum\limits_{i=1}^N \sum\limits_{p=1}^{r_{ik}}\phi_l(\gamma_{ip}^m(\lambda_1)\gamma_{ip}^m(\lambda_2)^*)=
\sum\limits_{i=1}^N \sum\limits_{j=1}^M \sum\limits_{p=1}^{r_{ik}} \sum\limits_{q=1}^{s_{ji}}
\lambda_{jq}^l(\gamma_{ip}^m(\lambda_1))\lambda_{jq}^l(\gamma_{ip}^m(\lambda_2))^*$$
that by (\ref{eq:Cond3}) equals
$$ \lambda_1\left(\sum\limits_{i=1}^N \sum\limits_{j=1}^M \sum\limits_{p=1}^{r_{ik}} \sum\limits_{q=1}^{s_{ji}}
\lambda_{ipq}^{jk}(\lambda_{ipq}^{jk})^*\right)\lambda_2^*$$
where the sums
ranges through the set
$$\bigsqcup\limits_{j=1}^M(w_kF^lw_j)=w_kF^l,$$
so that the above sum equals $j_{m, m+l}^F(\lambda_1\lambda_2^*)$, as
desired.\vspace{.2truecm}

Now, we proceed to define a partition for $F^{m+l+1} w_j$. For $\lambda\in
F^{m+1}w_k$,  we can define a partition
\begin{equation}\label{eq:Part5}
\lambda F^{l} w_j=\bigsqcup\limits_{i=1}^N \bigsqcup\limits_{\{ \gamma\in \Gamma_i^{l+1,\lambda}\}} \Lambda_j^{m+l+1, \gamma}
\end{equation}
where $\vert \Lambda_j^{m+l+1, \gamma} \vert= s_{ji}$ for $\gamma \in
\Gamma_i^{l+1,\lambda}$. Given any $\gamma=\gamma_{ip}^{m+1}(\lambda)\in
\Gamma_i^{l+1,\lambda}$, we will write
$$\Lambda_j^{m+l+1, \gamma}=\{ \lambda_{js}^{l+1}(\gamma)\mid 1\leq s\leq s_{ji}\}.$$

In analogy with (\ref{eq:Cond3}), for any $\lambda=\lambda_{kq}^1(f)\in F^{m+1} w_k$ we fix as definition
\begin{equation}\label{eq:Cond4}
\lambda_{js}^{l+1}(\gamma_{ip}^{m+1}(\lambda)):=\lambda \cdot \lambda_{ips}^{jk}.
\end{equation}

Now, we define $\phi_{l+1}: \mathcal M (E)_{l+1}\rightarrow \mathcal M (F)_{m+l+1}$ as follows: for any $\gamma_1, \gamma_2 \in E^{l+1}v_i$,
$$\phi_{l+1}(\gamma_1\gamma_2^*)=\sum\limits_{j=1}^M\sum\limits_{s=1}^{s_{ji}}\lambda_{js}^{l+1}(\gamma_1)\lambda_{js}^{l+1}(\gamma_2)^*.$$
It is easy to show that $\phi_{l+1}$ is a unital morphisms satisfying $K_0(\phi_{l+1})=S$.
A similar computation to the one given above shows that $\phi_{l+1}\cdot \psi_{m+1}= j_{m+1,m+l+1}^F$.

Now, let $\lambda \in F^m w_k$. Using $\textbf{(\mbox{Ex }2)}$, (\ref{eq:Cond4}) and (\ref{eq:Cond3}), we get
$$\lambda_{js}^{l+1}(\widehat{\gamma_{ip}^m(\lambda)})= \lambda_{js}^{l+1}(\gamma_{ip}^{m+1}(\widehat{\lambda}))=
\widehat{\lambda}\cdot \lambda_{ips}^{jk}= \widehat{\lambda \cdot \lambda_{ips}^{jk}}= \widehat{\lambda_{js}^l(\gamma_{ip}^m(\lambda))}.$$
Therefore, for $\gamma \in \Gamma _i^{l,\lambda}$, we have
$$\textbf{(\mbox{Ex }3)}\hspace{1truecm} \lambda_{js}^{l+1}(\widehat{\gamma})= \widehat{\lambda_{js}^l(\gamma )}.$$
As before, this implies that $\phi_{l+1}\cdot \alpha =\beta\cdot \phi_{l}$.

\subsection{The inductive argument}
\label{subsec:inductivearg}

We are ready to state the induction hypothesis and show the
inductive step. Note that everything is determined by the choices
we have made in the previous subsections.

Let $n\in \N$, let $\varepsilon = 0,1$, and
suppose that for any $1\leq k\leq n$ we have constructed:
\begin{enumerate}
\item \underline{Partitions}:
\begin{enumerate}
\item Given $\gamma\in E^{(k-1)l+\varepsilon}$, $1\le i \le N$, we have
$$\gamma E^l v_i = \bigsqcup\limits_{j=1}^M \bigsqcup\limits_{\{ \lambda \in \Lambda_j^{m+(k-1)l+\varepsilon, \gamma }\}}\Gamma_i^{kl+\varepsilon, \lambda},$$
with $| \Gamma_i^{kl+\varepsilon, \lambda } |= r_{ij}$ for all $\lambda \in \Lambda_j^{m+(k-1)l+\varepsilon, \gamma}$.
The elements of the subsets are denoted
$$\Gamma _i^{kl+\varepsilon, \lambda }= \{ \gamma_{ir}^{(k-1)l+m+\varepsilon}(\lambda ) \mid 1\leq r\leq r_{ij}\},$$
where $\lambda =\lambda_{jq}^{(k-1)l+\varepsilon}(\gamma )$, and are defined by the rule
\begin{equation}\label{Eq:(5.9')}
\gamma_{ir}^{(k-1)l+m+\varepsilon}(\lambda_{jq}^{(k-1)l+\varepsilon}(\gamma)):=
\gamma \cdot \gamma_{ir}^{m}(\lambda_{jq}^{0}(r(\gamma ))).
\end{equation}
\item Given $\lambda\in F^{(k-1)l+m+\varepsilon} w_t$, $1\le j\le M$, we have
$$\lambda F^l w_j=\bigsqcup\limits_{i=1}^N \bigsqcup
\limits_{\{ \gamma \in \Gamma _i^{kl+\varepsilon, \lambda }\}} \Lambda_j^{kl+m+\varepsilon, \gamma},$$
with $| \Lambda_j^{kl+m+\varepsilon, \gamma } |= s_{ji}$ for all $\gamma \in \Gamma_i^{kl+\varepsilon, \lambda}$.
The elements of the subsets are denoted
$$\Lambda _j^{kl+m+\varepsilon, \gamma}= \{ \lambda_{jq}^{kl+\varepsilon}(\gamma) \mid 1\leq q\leq s_{ji}\},$$
where $\gamma=\gamma_{ip}^{(k-1)l+m+\varepsilon}(\lambda )$, and are defined by the rule
\begin{equation}\label{Eq:(5.12')}
\lambda_{jq}^{kl+\varepsilon}(\gamma_{ip}^{(k-1)l+m+\varepsilon}(\lambda)):=
\lambda \cdot \lambda_{ipq}^{jt}.
\end{equation}

\end{enumerate}

\item \underline{Exchange identities}:
\begin{enumerate}
\item Given $\lambda\in \Lambda_j^{m+(k-1)l,\gamma }$, we have
$$\textbf{(\mbox{Ex }2k+1)}\hspace{1truecm} \gamma_{ir}^{m+(k-1)l+1}(\widehat{\lambda })= \widehat{\gamma_{ir}^{m+(k-1)l}(\lambda)}.$$
\item Given $\gamma\in \Gamma_i^{kl,\lambda}$, we have
$$\textbf{(\mbox{Ex }2k)}\hspace{1truecm} \lambda_{js}^{kl+1}(\widehat{\gamma})= \widehat{\lambda_{js}^{kl}(\gamma)}.$$
\end{enumerate}

\item \underline{Unital homomorphisms}:
\begin{enumerate}
\item $\psi_{m+(k-1)l+\varepsilon}:\mathcal M (F)_{m+(k-1)l+\varepsilon}\rightarrow \mathcal M (E)_{kl+\varepsilon}$ defined by the rule
$$\psi_{m+(k-1)l+\varepsilon}(\lambda_1 \lambda_2^*)=\sum\limits_{i=1}^N \sum\limits_{r=1}^{r_{ij}}\gamma_{ir}^{m+(k-1)l+\varepsilon}(\lambda_1)
\gamma_{ir}^{m+(k-1)l+\varepsilon}(\lambda_2)^*,$$
where $r(\lambda_1)=r(\lambda_2)=w_j$.
\item $\phi_{kl+\varepsilon}:\mathcal M (E)_{kl+\varepsilon}\rightarrow \mathcal M (F)_{m+kl+\varepsilon}$ defined by the rule
$$\phi_{kl+\varepsilon}(\gamma _1\gamma_2^*)=\sum\limits_{j=1}^M \sum\limits_{q=1}^{s_{ji}}\lambda_{jq}^{kl+\varepsilon}(\gamma_1)
\lambda_{jq}^{kl+\varepsilon}(\gamma_2)^*,$$ where
$r(\gamma_1)=r(\gamma_2)=v_i$.
\end{enumerate}
\end{enumerate}

\underline{satisfying}:
\begin{enumerate}
\item $K_0(\psi_{m+(k-1)l+\varepsilon})=R$ and $K_0(\phi_{kl+\varepsilon})=S$.
\item $\psi_{m+(k-1)l+\varepsilon}\cdot \phi_{(k-1)l+\varepsilon} = j_{(k-1)l+\varepsilon, kl+\varepsilon}^E$ and $\phi_{kl+\varepsilon}\cdot \psi_{m+(k-1)l+\varepsilon}= j_{m+(k-1)l+\varepsilon, m+kl+\varepsilon}^F$.
\item $\alpha\cdot \psi_{m+(k-1)l}=\psi_{m+(k-1)l+1}\cdot \beta$ and $\phi_{kl+1}\cdot \alpha =\beta\cdot \phi_{kl}$.
\end{enumerate}

Now, we will show that the same holds for $k=n+1$. First, we will take the commutative diagram

$$\xymatrix{\Z^N\ar[rr]^{\Omega _{(n+1)l}}  & & \Z^N\\
 & & \\
\Z^M\ar[rr]_{\Omega '_{m+nl}} \ar[uu]_{R}& & \Z^M\ar[uu]^{R}}
$$

and we will lift it to a commutative diagram

$$\xymatrix{\mathcal M (E)_{(n+1)l}\ar[rr]^{\alpha} & & \mathcal M (E)_{(n+1)l+1}\\
 & & \\
\mathcal M (F)_{m+nl}\ar[rr]_{\beta} \ar[uu]_{\psi_{m+nl}}& & \mathcal M (F)_{m+nl+1} \ar[uu]^{\psi_{m+nl+1}}}.
$$

We now define a partition of $E^{(n+1)l+\varepsilon} v_i$. For
$\gamma \in E^{nl+\varepsilon}v_k$, we can take a partition
\begin{equation}\label{eq:Part6}
\gamma E^lv_i = \bigsqcup\limits_{j=1}^M
\bigsqcup\limits_{\{ \lambda\in \Lambda_j^{m+nl+\varepsilon, \gamma}\}} \Gamma_i^{(n+1)l+\varepsilon, \lambda}
\end{equation}
where $\vert \Gamma_i^{(n+1)l+\varepsilon, \lambda}\vert =r_{ij}$ for all $\lambda \in \Lambda_j^{m+nl+\varepsilon, \gamma}$.
Given any $\lambda =\lambda_{jq}^{nl+\varepsilon}(\gamma )\in \Lambda_j^{m+nl+\varepsilon, \gamma}$, we will denote
$$\Gamma_i^{(n+1)l+\varepsilon, \lambda }=\{ \gamma _{ir}^{m+nl+\varepsilon}(\lambda )\mid 1\leq r\leq r_{ij}\}.$$
In analogy with (\ref{eq: Cond2}), we can fix as definition
\begin{equation}\label{eq:(5.9'')}
\gamma_{ir}^{m+nl+\varepsilon}(\lambda_{jq}^{nl+\varepsilon}(\gamma )):=\gamma \cdot
\gamma_{ir}^{m}(\lambda_{jq}^0(r(\gamma ))).
\end{equation}

Now, we take $\gamma \in E^{nl}v_k$, and we compute, using (Ex 2n) and (\ref{eq:(5.9'')}) that
$$\gamma_{ir}^{m+nl+1}(\widehat{\lambda_{jq}^{nl}(\gamma )})= \gamma_{ir}^{m+nl+1}(\lambda _{jq}^{nl+1}(\widehat{\gamma}))
= \widehat{\gamma}\cdot \gamma _{ir}^m(\lambda_{jq}^0 (r(\gamma)))= \widehat{\gamma_{ir}^{m+nl}(\lambda_{jq}^{nl}(\gamma))} .$$
This gives that for all $\lambda\in \Lambda_j^{m+nl,\gamma }$,
$$\textbf{(\mbox{Ex }2n+1)}\hspace{1truecm} \gamma_{ir}^{m+nl+1}(\widehat{\lambda}) = \widehat{\gamma_{ir}^{m+nl}(\lambda)} . $$

In the same way as in the previous cases, we define, for $\lambda_1, \lambda_2 \in F^{m+nl+\varepsilon}w_j$
$$\psi_{m+nl+\varepsilon}(\lambda_1 \lambda_2^*)=\sum\limits_{i=1}^N \sum\limits_{r=1}^{r_{ij}}
\gamma_{ir}^{m+nl+\varepsilon}(\lambda_1)\gamma_{ir}^{m+nl+\varepsilon}(\lambda_2)^*.$$
It is easy to show that $\psi_{m+nl+\varepsilon}$ is a unital
morphism such that $K_0(\psi_{m+nl+\varepsilon})=R$.

Given $\gamma_1,\gamma _2\in E^{nl+\varepsilon} v_k$, one checks, using (\ref{eq:(5.9'')}) and (\ref{eq:Part2(ii)}) that
$$\psi_{m+nl+\varepsilon}\cdot \phi_{nl+\varepsilon}(\gamma_1\gamma_2^*)=
 \gamma_1\left( \sum\limits_{\gamma\in v_kE^l} \gamma \gamma^*\right) \gamma _2^*= j_{nl+\varepsilon,(n+1)l+\varepsilon}^E(\gamma _1\gamma_2^*).$$

Finally, (Ex 2n+1)  gives that $\psi_{m+nl+1}\cdot \beta = \alpha \cdot \psi_{m+nl}$.

The last step of the proof, consisting in defining the sets
$\Lambda_j^{m+(n+1)l+\epsilon, \gamma}$ and the maps $\phi
_{(n+1)l+\epsilon}$ satisfying the appropriate relations, is similar
to the case $n=0$ which we did before in detail. We leave the details
to the reader. \hspace{\fill} $\Box$\vspace{.3truecm}

\section{Uniqueness of liftings}
\label{sect:uniqliftings}

 Let us recall the strong classification conjecture, due to Hazrat:

\begin{conjecture}
\label{conj3:Hazrat}{\rm \cite[Conjecture 3]{Hazrat} The graded Grothendieck group $K_0^{{\rm gr}}$
is a fully faithful functor from the category of unital Leavitt path algebras with graded homomorphisms
modulo inner-automorphisms to the category of pre-oredered abelian groups with order-unit.}
\end{conjecture}

\smallskip

The conjecture has been proved by Hazrat for acyclic graphs (\cite[Theorem 1]{Hazrat}).
The conjecture means that given any order-preserving $\Z[t,t^{-1}]$-isomorphism
$f\colon K_0^{{\rm gr}}(L(E))\to K_0^{{\rm gr}} (L(F)) $, such that $f([R])=[S]$, there exists a graded
isomorphism $\psi \colon R\to S$ such that $K_0(\psi )= f$. Moreover, given any two such isomorphisms
$\varphi$ and $\psi$, with $K_0(\psi) = K_0(\varphi)$,  there is an inner automorphism $\tau$ of $S$ such that
$\psi = \tau \circ \varphi$.

\smallskip

In this section, we will study the uniqueness part of this conjecture, obtaining, for any finite graph
without sources $E$, the form of the automorphisms $\theta $ of $L(E)$ such that $K_0^{\rm{gr}}(\theta ) = \rm{id}$.
Their form suggests that they are not necessarily inner automorphisms, and indeed this class
of automorphisms always includes the ones given by the torus action (or gauge action, in the language of
C*-algebras), which are not inner in general. In order to get a complete description of this group of automorphisms,
it remains to elucidate when certain injective algebra endomorphisms are surjective.

\smallskip

 Let $\mathfrak G $ be the group of graded automorphisms $\varphi $ of $L(E)$ such that
 the restriction of $\varphi $ to $\mathcal M (E) $ is locally inner. We first observe that this is the kernel
 of the map $\text{gr-Aut}(L(E))\to \text{Aut}_{\Z [t,t^{-1}]}(K_0^{{\rm gr}}(L(E)))$.

 \smallskip

 We need a lemma:

\begin{lem}
 \label{lem:transl-along-tails}
 Let $F$ be a row-finite graph having a tail $v_0,v_1,v_2,\dots $, that is, for each $i \ge 0$, there is exactly one arrow $e_i$
 such that $s(e_i)=v_i$, and in addition we have $r(e_i)=v_{i+1}$. Then $v_{i+1}L(F)\cong \mathcal T (v_iL(F))$, where $\mathcal T$ denotes
 the translation functor.
\end{lem}

\begin{proof}
 Consider the graded isomorphism $\gamma \colon v_{i+1}L(F)\to \mathcal T (v_{i}L(F))$ given by $\gamma (v_{i+1}x)= e_ix$.
 \end{proof}

 \begin{prop}
  \label{prop:kernel-map}
  Let $E$ be a finite graph. Then the group
 $\mathfrak G $ of graded $K$-algebra automorphisms of $L(E)$ whose restriction to $\mathcal M (E) $ is locally inner
is precisely the group of graded $K$-algebra automorphisms $\varphi $ of $L(E)$ such that $K_0^{{\rm gr}}(\varphi )= {\rm id}$.
 \end{prop}

\begin{proof}
 If $\varphi$ is a graded automorphism of $L(E)$ such that  $K_0^{{\rm gr}}(\varphi )= {\rm id}$, then the restriction $\varphi _0$ of $\varphi $
 to $\mathcal M (E)$ must satisfy $K_0(\varphi _0)=\text{id}$, and so $\varphi _0$ must be locally inner.

 Conversely, let $\varphi \in \mathfrak G$. It suffices to show that the graded finitely generated projective $L(E)$-modules
 are direct sums of translates of projective modules of the form $P\otimes_{\mathcal M (E)} L(E)$, where $P$ is a finitely generated projective
 $\mathcal M (E)$-module. This is clear for finite graphs without sinks, because then, by Dade's Theorem and \cite[Theorem 4]{Hazrat}, every graded
 finitely generated projective $L(E)$-module is induced by a finitely generated projective $\mathcal M (E)$-module.
 If $E$ does have sinks then we apply the desingularization process to obtain a new graph $F$ without sinks such that $L(E)$ and $L(F)$ are Morita-equivalent.
 Indeed there is a homogeneous idempotent $e$ in $L(F)$ such that $L(E)$ is graded-isomorphic to $eL(F)e$, and such that
 $L(F)eL(F)=L(F)$. It follows from \cite[Example 2]{Hazrat2} that there is a graded equivalence between $gr-L(E)$ and $gr-L(F)$.
 Now $L(F)$ is strongly graded (by \cite[Theorem 4]{Hazrat}), so by Dade's Theorem its graded finitely generated projective modules come from
 the finitely generated projective $\mathcal M (F)$-modules. The finitely generated projective $\mathcal M (F)$-modules
 are generated by the finitely generated projective modules of the form $P\otimes _{\mathcal M (E)} e\mathcal M (F) $,
 for $P$ a finitely generated projective $\mathcal M (E)$-module, and the projective modules of the form $v_i\mathcal M (F)$,
 where $v_0$ is a sink in $E$ and $v_1,v_2, \dots , $ is the tail added to $v_0$ in $F$.  Now observe that, by Lemma \ref{lem:transl-along-tails},
 we have $v_iL(F)\cong \mathcal T ^i(v_0L(F))$, where $\mathcal T$ denotes the translation functor.
 Therefore the graded finitely generated projective $L(F)$-modules are direct sums of translates of
 graded finitely generated projective modules of the form $P\otimes_{\mathcal M (E)} eL(F)$, where $P$ is a  finitely generated
 projective $\mathcal M (E)$-module, and therefore the same is true for the graded finitely generated projective $L(E)$-modules.
\end{proof}

\smallskip

Let $E$ be a finite graph without sources, and let $u, z$ be invertible elements in
$\mathcal M(E)$ such that $z(uvu^{-1})= (uvu^{-1})z$ for all $v\in E^0$. We
define the algebra endomorphism $\theta _{u,z}$ of $L(E)$ be the rules
$\theta_{u,z} (v)=uvu^{-1}$ for $v\in E^0$, and $\theta _{u,z}(e)=
(ueu^{-1}) z$, $\theta_{u,z}(e^*)=z^{-1}(ue^*u^{-1})$ for $e\in
E^1$. Then it is easily checked that $\theta _{u,z}$ induces a
graded algebra endomorphism of $L(E)$. It is clearly injective, because it is graded and $\theta _{u,z} (v) \ne 0$ for all $v\in E^0$
(see \cite[Theorem 4.8]{Tomforde}).
It is surjective if and only if $u^{\pm}$ and $z^{\pm}$ belong to the image of $\theta_{u,z}$. More concretely, if $s,t\in \mathcal{M}(E)$
are elements such that $\theta_{u,z}(s)=u^{-1}$ and $\theta_{u,z}(t)=u^{-1}z^{-1}u$, then $s,t$ are invertible, satisfy $(sxs^{-1})t=t(sxs^{-1})$ for every $x\in E^0$, and $\theta_{s,t}=\theta_{u,z}^{-1}$.

\begin{lem}
\label{lem:K_0-trivial}
With the above notation, we have $K_0^{\rm{gr}}(\theta _{u,z}) = \rm{id}$.
\end{lem}

\begin{proof} By Proposition \ref{prop:kernel-map}, it suffices to check that the restriction of $\theta_{u,z}$ to
$\mathcal M (E)$ is a locally inner endomorphism of $\mathcal M (E)$.

 Let $Q_n$ be the set of paths $\gamma $ such that
either $|\gamma |= n$ or $|\gamma |< n$ and $r(\gamma)$ is a sink.
We will use the filtration $\mathcal M (E) =  \bigcup _{n=0}^ {\infty} \mathcal M (E)_n$ of $\mathcal M (E)$,
where $\mathcal M (E)_n$ is the linear span of the elements $\gamma \mu ^*$, where $\gamma, \mu \in Q_n$,
$|\gamma | =|\mu |$ and $r(\gamma ) = r(\mu)$.

 For $v\in E^0$, define $u^v= uv\in \mathcal M (E)$ and, for $e\in E^1$, define $u^e= uee^*\in \mathcal M (E)$. Furthermore, for
 $ \gamma = e_1e_2\cdots e_n\in E^n$ with $n\ge 2$, define
$$u^{\gamma}= u^{e_1e_2\cdots e_n} = (ue_1u^{-1})z(ue_2u^{-1})z \cdots (ue_{n-1}u^{-1})z(ue_n)\gamma ^* \in \mathcal M (E) .$$
Finally define, for each $n\ge 0$,
$$u_n = \sum _{\gamma \in Q_n} u^{\gamma} .$$
It is easy to check that $u_n$ is an invertible element in $ \mathcal M(E)$, with inverse given by
$$u_n^{-1} = \sum _{\gamma \in Q_n} (u^{\gamma})^\dag ,$$
where $(u^v)^\dag = vu^{-1}$ for $v\in E^0$, $(u^e)^\dag = ee^*u^{-1}$ for $e\in E^1$, and
$$(u^\gamma)^\dag = \gamma (e_n^*u^{-1})z^{-1}(ue_{n-1}^*u^{-1})\cdots z^{-1}(ue_1^*u^{-1})$$
for $\gamma = e_1 \cdots e_n\in E^n$ with $n\ge 2$.

One can check that, for $x\in \mathcal M (E)_n$, we have
$$\theta_{u,z} (x) = u_nxu_n^{-1} \, ,$$
which shows that $\theta _{u,z}$ is a locally inner endomorphism of $\mathcal M (E)$. It follows that
 $K_0^{\rm{gr}}(\theta _{u,z}) = \rm{id}$, as desired.
 \end{proof}

\begin{theor}
\label{theor:realizingiso} Let $E$ and $F$ be finite graphs without
sources. Let $\varphi $ and $\psi$ be two graded algebra isomorphisms from
$L_K(E)$ onto $L_K(F)$ such that $K_0^{{\rm gr}}(\varphi)=K_0^{{\rm gr}}(\psi)$. Then there exist invertible elements $u,z$
in $\mathcal M (F)_0$, with $z(uvu^{-1})= (uvu^{-1})z$ for all $v\in F^0$,
such that
$$\psi = \theta _{u,z} \circ \varphi .$$
\end{theor}

\begin{proof}
It is enough to determine the structure of the graded algebra
automorphisms $\varphi $ of $L(E)$ such that $K_0^{gr}(\varphi)
=\mbox{id}$. So, let $\varphi$ be such an automorphism, and denote by
$\varphi _0$ the restriction of $\varphi $ to $\mathcal M (E)$, which is an
algebra automorphism of the ultramatricial algebra $\mathcal M (E)$ such
that $K_0(\varphi _0)= \mbox{id}$. By \cite[Lemma 15.23(b)]{vnrr},
there exists $u\in \mathcal M (E)$ such that $\varphi (x)=uxu^{-1}$ for all
$x\in \mathcal M (E)_{1}$.

Set $A=\mathcal M (E)$, set $t_+= \sum _{i=1} ^n e_i$, $t_-=\sum _{i=1} ^n
e_i ^*$, where $e_i \in E^1$ satisfy that $r(e_i) = v_i$ for all
$i$, and $E^0=\{v_1,\dots , v_n \}$. As observed before, we have
$L(E)=A[t_+,t_-, \alpha ]$, where $\alpha \colon A \to pAp$ is the
corner-isomorphism given by $\alpha (x) = t_+xt_-$ for $x\in A$.

It is easy to show that $\varphi $ is determined by $\varphi _0$ and
$\varphi (t_+)= a_0t_+$, $\varphi (t_-) = t_- b_0$, where $a_0\in
Ap$ and $b_0\in pA$ must satisfy the relations $b_0a_0= p$ and
$$\varphi _0 (\alpha (x))= a_0 \alpha (\varphi _0(x))b_0$$
for all $x\in A$. In particular we have $a_0b_0= \varphi _0(p)=
upu^{-1}$.

Set $ \zeta = \alpha (u) u^{-1} a_0 \in pAp$, and $\varpi =
b_0u\alpha (u^{-1})$. Then, we have
$$\zeta \varpi = \alpha (u) u^{-1} a_0b_0 u \alpha (u^{-1}) =
\alpha (u) u^{-1}upu^{-1} u \alpha (u^{-1})= \alpha (1) = p$$ and
$$\varpi \zeta = b_0 u\alpha (u^{-1}) \alpha (u) u^{-1} a_0 = b_0
u\alpha (1) u ^{-1}a_0= b_0 \varphi _0 (\alpha (1))a_0 = \alpha
(\varphi _0 (1)) = p.$$ So $\zeta $ is invertible in $pAp$, with
inverse $\varpi$. Set $z= \alpha ^{-1} (\zeta )$, an invertible
element of $A$. Then, for $v\in E^0$ we have
\begin{align*}
z(u & vu^{-1}) z^{-1} =\alpha ^{-1} (\zeta \alpha (\varphi
_0(v))\varpi )= \alpha^{-1} ( \alpha (u) u^{-1}a_0\alpha (\varphi
_0(v))b_0 u\alpha (u^{-1}))\\
& = \alpha^{-1} ( \alpha (u) u^{-1}\varphi _0(\alpha (v)) u\alpha
(u^{-1})) = \alpha^{-1}(\alpha (u)\alpha (v)\alpha( u^{-1}))=
uvu^{-1}\, ,
\end{align*}
showing that $z(uvu^{-1})= (uvu^{-1})z$ for all $v\in E^0$.
Moreover, if $e\in E^1$, then
\begin{align*} \varphi (e) & = \varphi_0
(et_-) a_0t_+ = uet_-u^{-1} a_0t_+= uet_-(pu^{-1} a_0)t_+ \\ & =
uet_-\alpha (u^{-1})\zeta t_+   = ueu^{-1}\alpha ^{-1}(\zeta ) =
ueu^{-1} z .
\end{align*}
 Similarly $\varphi
(e^*)= z^{-1} ue^*u^{-1}$. Since $\varphi $ and $\theta _{u,z}$
agree on the generators of $L(E)$, we conclude that $\varphi =
\theta _{u,z}$, as desired.
\end{proof}

Let $u,z$ be invertible elements
in $\mathcal M (E)_0$, with $z(uvu^{-1})= (uvu^{-1})z$ for all $v\in F^0$.
It is not clear whether all the endomorphisms $\theta _{u,z}$ are surjective.
As we are going to show, this question has a positive answer when, in addition, we have that $z\in Z(\mathcal M(E))$.

We denote by $U(R)$ the group of units of a unital ring $R$.

\begin{prop} Let $E$ be a finite graph without sources.
 Let $\mathfrak G $ be the group of graded automorphisms $\varphi $ of $L(E)$ such that
 the restriction of $\varphi $ to $\mathcal M (E) $ is locally inner. Then
 there is a group homomorphism
 $$\zeta \colon U(\mathcal M (E))\times U(Z(\mathcal M(E)))  \to \mathfrak G .$$
 The kernel of $\zeta$ is the set of elements $(u,z)$ in $U(\mathcal M (E))\times U(Z(\mathcal M(E)))$
 such that $u\in Z(\mathcal M (E))$ and $ueu^{-1} = ez^{-1}$ for all $e\in E^1$.
 \end{prop}

\begin{proof}
Define $\zeta  \colon U(\mathcal M (E))\times U(Z(\mathcal M(E))) \to \mathfrak G $ by $\zeta (u,z)= \theta _{u,z}$.
It is easy to check that for $(u,z)\in U(\mathcal M (E))\times U(Z(\mathcal M(E)))$ we have
$\theta _{u,z }(x) = uxu^{-1}$ for all $x\in \mathcal M (E)$.
Using this it is straightforward to check that $\theta _{u_1,z_1 }\circ \theta _{u_2,z_2}= \theta _{u_1u_2, z_1z_2}$
for $(u_1,z_1), (u_2,z_2) \in U(\mathcal M (E))\times U(Z(\mathcal M(E)))$.
This shows that $\zeta $ is a group homomorphism. If $(u,z)$ belongs to the kernel of $\zeta$, then
$u$ must belong to the center of $\mathcal M (E)$ by the previous observation, and $ueu^{-1}= ez^{-1}$.
Conversely if these conditions are satisfied then for $e\in E^1$ we have
$$z e^*= (ze^*) (ueu^{-1}(ue^*u^{-1}) = z(e^*e)z^{-1} (ue^*u^{-1}) = e^*eue^*u^{-1}= ue^*u^{-1}\, ,$$
so that $ue^*u^{-1} = ze^*$ and we conclude that $\theta _{u,z}=  \text{id}$.
\end{proof}

We now compute the group $\mathfrak G$ in a particular example. This gives a counterexample to the
uniqueness part of \cite[Conjecture 3]{Hazrat}, in the sense that we obtain a non-inner graded automorphism $\varphi $ of
a Leavitt path algebra such that $K_0^{gr}(\varphi) = \text{Id}$, due to the presence of the torus action. Indeed, the same
obstruction to innerness appears already for the Leavitt path
algebra of the single loop $K[x,x^{-1}]$.

\begin{exem}{\rm
 Let
 $$\mathfrak T := K\langle x,y \mid yx=1 \rangle $$
 be the (algebraic) Toeplitz algebra, which can be realized
 as the Leavitt path algebra associated to the graph $E$ with $E^0=\{ v,w \}$
 and $E^1=\{\alpha, \beta \}$, with $s(\alpha ) = r(\alpha ) = v= s(\beta )$, and $r(\beta ) = w$.
 The algebra $L(E)$ is isomorphic to $\mathfrak T$ via an isomorphism that sends $\alpha +\beta $ to $x$ and $\alpha^*+\beta^*$
 to $y$.

 Define $e_{ij}= x^i(1-xy)y^j$ for $i,j\ge 0$. This is a set of matrix units in $\mathfrak T $, the ideal $I=\bigoplus _{i,j} e_{ij}K\cong M_{\infty} (K)$
 is the only non-trivial ideal of $\mathfrak T$ and $\mathfrak T/I\cong K[x,x^{-1}]$.
The $0$-component of $\mathfrak T$ is precisely
 $$\mathfrak T _0 =\{ \sum _{i=0}^n a_ie_{ii}+a_{n+1}x^{n+1}y^{n+1} \mid a_0,\dots ,a_{n+1}\in K, n\ge 0 \}\cong K\cdot 1 \oplus (\bigoplus _{\Z^+} K ).$$
 In particular, $\mathfrak T _0$ is a commutative algebra, and so it follows from Theorem \ref{theor:realizingiso} that
 $\mathfrak G \cong \zeta (U(\mathfrak T _0)\times U (\mathfrak T _0))$.
 Now, it is shown in \cite[Theorem 4.3]{Bavula} that
 $$\text{Aut}_{K-\text{alg}}(\mathfrak T) \cong \mathbb T \rtimes \text{Inn}(\mathfrak T) \cong \mathbb T \rtimes GL_{\infty} (K).$$
 Here $\mathbb T$ denotes the ``torus action'' of $K^{\times}$ on $\mathfrak T$, which, for $a\in K^{\times}$, sends $x$ to $ax$ and
 $y$ to $a^{-1}y$. Note that this is precisely $\zeta ((1,a\cdot 1 ))\in \mathfrak G$. The group $GL_{\infty}(K)$ is the group of units
 of the multiplicative monoid $1+M_{\infty}(K)$, and it acts by conjugation on $\mathfrak T$. Denote by $D_{\infty} (K)$ the group of units of
 the monoid $1+(\bigoplus _{i=0}^{\infty} e_ {ii}K)$. For $u = \sum _{i=0}^n a_ie_{ii}+x^{n+1}y^{n+1}\in D_{\infty}(K)$, we have the formula
 $$uxu^{-1} = x(\sum _{i=0}^n a_i^{-1} a_{i+1}e_{ii} + x^{n+1}y^{n+1} ) ,$$
 (where $a_{n+1}=1$). It follows from this that
 $$\zeta (1,D_{\infty}(K)) = \zeta (U(\mathfrak T _0), 1) = D_{\infty}(K)\subset GL_{\infty}(K) \subseteq \text{Inn} (\mathfrak T ) ,$$
 which shows that
 $$\mathfrak G = \mathbb T \times D_{\infty}(K).$$
An element in $\mathfrak G $ is an inner automorphism of $\mathfrak T$ if and only its $\mathbb T$-component is $1$.
}
\end{exem}

\section*{Acknowledgments}

Part of this work was done during a visit of the second author to the Centre de Recerca Matem\`atica
(U.A.B., Spain) in the context of the Research Program ``The Cuntz semigroup and the
classification of $C^*$-algebras''. The second author wants to thank the host center for its
warm hospitality, and also to George Elliott for some interesting discussions about the topic of the paper.

 We also thank the anonymous referee for carefully reading the paper and for providing many helpful suggestions.

\end{document}